\documentclass[10pt,a4paper,reqno]{amsart}
\usepackage[latin1]{inputenc}
\usepackage{amsmath}
\usepackage{amssymb} 
\usepackage{amsfonts, verbatim,tikz, hyperref,bbold, bbm}
\usepackage[all,cmtip]{xy}

\newtheorem{theorem}{Theorem}
\newtheorem{corollary}{Corollary}
\newtheorem{lemma}{Lemma}

\newtheorem{prop}{Proposition}
\newtheorem{defn}{Definition}

\theoremstyle{definition}
\newtheorem{remark}{Remark}

\newcommand{\C}{\mathbb C}
\newcommand{\ztwo}{\mathbb Z/2\mathbb Z}
\newcommand{\nc}{\newcommand}

\nc{\thm}{\theorem}
\nc{\cor}{\corollary}
\nc{\mc}{\mathcal}
\nc{\mb}{\mathbb}
\nc{\mf}{\mathfrak}
\nc{\ul}{\underline}
\nc{\ol}{\overline}
\nc{\N}{\mb N}
\nc{\R}{\mb R}
\nc{\Z}{\mb Z}
\nc{\Q}{\mb Q}
\nc{\F}{\mb F}
\nc{\gm}{\gamma}
\nc{\Hom}{Hom}
\nc{\SO}{SO}
\nc{\Spin}{Spin}
\nc{\X}{\mf{X}}
\nc{\I}{\text{I}}
\nc{\J}{\text{J}}
\nc{\lip}{\langle}
\nc{\rip}{\rangle}
\nc{\dmo}{\DeclareMathOperator}
\dmo{\GL}{GL}
\dmo{\sgn}{sgn}
\dmo{\res}{res}
\dmo{\diag}{diag}
\dmo{\pr}{pr}

\title{Total Stiefel Whitney classes for  real representations  of $\GL_n$ over $\F_q, \R$ and $\C$ }
\author{Jyotirmoy Ganguly, Rohit Joshi}
\makeatletter\@addtoreset{chapter}{part}\makeatother%
\address{Chennai Mathematical Institute, H1, SIPCOT IT Park, Siruseri, Kelambakkam 603103, Tamil Nadu, India} \email{jyotirmoy.math@gmail.com}  

\address{IISER Pune, Dr. Homi Bhabha Road,Pune-411008, Maharashtra, India} \email{rohitsj@students.iiserpune.ac.in}  

\begin{document}
	
	\begin{abstract}
We compute the total Stiefel Whitney class for a real representation $\pi$ of $\GL_n(\F_q)$, where $q$ is odd in terms of character values of $\pi$ on order $2$ diagonal elements. We also compute the total Stiefel Whitney classes of real representations of $\GL_n(\R)$ and $\GL_n(\C)$.
	\end{abstract}	
	
	\maketitle

\tableofcontents

\section{Introduction}

Let $\pi$ be a real representation of a finite group $G$. One can define Stiefel Whitney classes (SWC) $w_i(\pi)$ for $0\leq i\leq \deg\pi$ as members of the cohomology groups $H^i(G)=H^i(G,\ztwo)$. The total Stiefel Whitney class of $\pi$ in the cohomology ring $H^*(G)$ is defined as
$$w(\pi)=w_0(\pi)+w_1(\pi)+\cdots+w_d(\pi),$$
where $d=\deg\pi$. (see \cite[Section 2.6, page no. 50]{benson} for details). 

In this article, we compute the total SWC for a real representation $\pi$ of $\GL_n(\F_q)$, where $q$ is odd, in terms of the character values on diagonal elements of order $2$. The computations for the particular case of $n=2$ can be found in \cite{ganjo}.
We also compute the total SWC of a real representation of $\GL_n(\C)$ and $\GL_n(\R)$. The computation of second SWC for representations of $S_n$ and related groups can be found in \cite{jyoti}. Explicit expressions for the first and second SWC of the representations of $\mathrm{O}_n(\R)$ were found in \cite{ganjo1} in terms of highest weights. 

We will now describe our results in some more detail.
Let $\mathfrak{X}$ denote the set of binary strings of length $n$. For $a=(a_1,a_2, \ldots ,a_n)\in \mathfrak{X}$, define $\nu(a)= \sum\limits_{i=1}^n a_i \in \Z$. Put $\mathfrak{X}_i=\{a\in\mathfrak{X}\mid \nu(a)=i\}$. 
Let $C_m$ denote the additive cyclic group of order $m$. We fix a generator $x$  for $C_m$.

Consider the linear character $\chi^j$ of $C_m$ given by $\chi^{j}(x)=\zeta_m^j$, where $\zeta_m$ is the primitive $m$th root of unity. For $m$ even, write `$\sgn$' for $\chi^{m/2}$.
For a binary string $a\in \X$ and $m$ even, consider the representation $\sgn_a$ of $C_m^n$, defined as $\sgn_a = \boxtimes_{j=1}^n (\sgn)^{a_j}$ (external tensor product), where
\[
(\sgn)^{\epsilon}=
\begin{cases}
\mathbb{1},\quad \text{if $\epsilon=0$},\\
\sgn,\quad \text{if $\epsilon=1$}.
\end{cases}
\]

Here we use some terminologies from supercharacter theory. For more details we refer the reader to \cite{super}. Consider the action of $S_n$ on $C_2^n$ by permuting the copies of $C_2$. If we identify $C_2^n$ with set of binary strings $\X$, then the superclasses are $\X_i$ which are orbits of $S_n$ action. 

We will identify $C_2^n$ with the subgroup $$\text{diag}(\pm 1, \ldots, \pm 1) \subset \GL_n(F),$$ where $F=\F_q, \R$ or $\C$ depending upon the context. We will write 
 $$h_i=\diag(\underbrace{-1,\ldots,-1}_{i \, \text{times}},1,\ldots,1)\in\GL_n(F).$$
Then $h_i$ represents the superclass $\X_i$.
Consider the representations 
\begin{equation}\label{sig}
	\sigma_i=\bigoplus\limits_{l\in \X_i}\sgn_l
\end{equation}
of $C_m^n$ for $1\leq i\leq n$. The supercharacters are the direct sum of the representations in a single orbit for the action of $S_n$ on characters of $C_2^m$. Then clearly the supercharacters are $\sigma_i$.
%The cohomology ring of $\GL_n(\F_q)$, where $q\equiv 1\pmod 4$ is given by
%\begin{equation}
%H^*(\GL_n(\F_q),\ztwo)=\ztwo[a_2,a_4,\ldots,a_{2n}]\otimes E[b_1,b_3,\ldots,b_{2n-1}],
%\end{equation} 
%where 
%$$a_{2l}=\sum_{i_1<i_2<\ldots<i_l}t_{i_1}\cdots t_{i_l},\quad b_{2l-1}=\sum_{i_1<i_2<\ldots<i_l, 1\leq j\leq l}t_{i_1}\cdots \hat{t}_{i_j}\cdots t_{i_l}s_{i_j},\quad 1\leq l\leq n.$$
%For details we refer the reader to [Priddy, Theorem $8.1$, page $288$]. 

For $l\in \X$, consider the linear character $\chi(l)$ of $C_m^n$ given by $\chi(l)=\bigotimes\limits_{k=1}^n\chi^{l_k}$. For $1\leq i\leq n$, take $e_i\in \X$  such that $e_{ij}=1$ if $i=j$ and $e_{ij}=0$ otherwise. If $4\mid m$ then 
$$H^*(C_m^n)=\dfrac{\ztwo[s_1,\ldots,s_n,t_1,\ldots,t_n]}{(s_1^2,\ldots,s_n^2)},$$
where $s_i=w_1(\sgn_{e_i})$ and $t_i=w_2(\chi(e_i)_{\R})$. On the other hand
$$H^*(C_2^n)=\ztwo[v_1,\ldots,v_n],$$
where $v_i=w_1(\sgn_{e_i})$.
Let $v=(v_1,\ldots,v_n)$ and $t=(t_1,\ldots,t_n)$ . For $l\in \X$ we write $l\cdot v$ (resp. $l\cdot t$) for $\sum\limits_{i=1}^{n}l_iv_i$ (resp. $\sum\limits_{i=1}^{n}l_it_i$).

 For a real representation $\pi$ of $\GL_n(F)$ where $F=\F_q$ or $\R$ or $\C$, consider the term 
$$c_k=\frac{1}{2^n}\sum\limits_{i=0}^{n}\chi_{\sigma_i}(h_k)\chi_{\pi}(h_i),$$ where $M=(\chi_{\sigma_i}(h_k))$ is the matrix which is exactly the supercharacter table. Thus we explicitly compute the table by proving that  $\chi_{\sigma_i}(h_k)$ is the coefficient of $y^i$ in the expression $(1-y)^k(1+y)^{n-k}$(see Proposition \ref{chi_i,k}). 

For a real representation $\pi$ of $\GL_n(\F_q)$ we define 
\[
\delta=
\begin{cases}
0,& \text{if $\det\pi=1$},\\
1, & \text{otherwise}.
\end{cases}
\] For a subgroup $H$ of $G$ we denote by $w^H(\pi)$ (resp. $w_i^H(\pi)$) the total Stiefel Whitney class $w(\pi\mid_{H})=\res^G_H(w(\pi))$(resp. $w_i(\pi\mid_H)$). 
Let $D$ denote the diagonal torus of $\GL_n(\F_q)$.
We have $\mathrm{res}:H^*(\GL_n(\F_q)) \hookrightarrow H^*(D)$ for $q\equiv 1 \pmod 4$ and $\mathrm{res}:H^*(\GL_n(\F_q)) \hookrightarrow H^*(C_2^n) $ for $q \equiv 3 \pmod 4$ are injections (see \cite[Theorem $4.4$, page $227$]{adem}). Our main theorem, proved in Section \ref{pmt}, says that
\begin{theorem}[Main Theorem]\label{main}
Let $\pi$ be a real representation of $\GL_n(\F_q)$ where $q$ is odd. 
\begin{enumerate}
\item
If $q\equiv 1\pmod4$ then 
$$w^D(\pi)=\left(1+ \delta b_1\right) \prod_{i=1}^n\left(\prod_{p \in \X_i}(1 + p\cdot t) \right)^{c_{i}/2},$$
where $p\cdot t=\sum\limits_{i=1}^np_it_i$ and $b_1=\sum\limits_{i=1}^ns_i$.
%where $m_\pi = (\dim \pi - \chi_{\pi}(h_1))/2$.
\vskip 2mm
\item
If $q\equiv 3\pmod 4$ then 
$$w^{C_2^n}(\pi)=\prod_{i=1}^{n}\left(\prod_{l\in\X_i}(1+l\cdot v)\right)^{c_i},$$
where $l\cdot v=\sum\limits_{i=1}^nl_iv_i$.

\end{enumerate}
\end{theorem}

We use the same approach to compute the total Stiefel Whitney class of real representations of $\GL_n(\C)$ and $\GL_n(\R)$.
%Let $C_2^n$ be the elementary abelian group of $2$ torsion elements of the diagonal torus of $\GL_n(k)$, where $k=\R$ or $\C$. 
We have $\res:H^*(\GL_n(k)) \hookrightarrow H^*(C_2^n)$ is injective when $k = \R \text{ or } \C$.
\begin{theorem}\label{RandC}
Let $\pi$ be a real representation of $\GL_n(\C)$ or $\GL_n(\R)$. Then the total Stiefel Whitney class of $\pi$ is given by
\begin{equation}\label{on}
w^{C_2^n}(\pi)=\prod_{i=1}^{n}\left(\prod_{l\in\X_i}(1+l\cdot v)\right)^{c_i}.
\end{equation}
\end{theorem}

%We say a subgroup $H$ of the group $G$ detects the mod $2$ cohomology of $G$ if the restriction map $\res:H*(G)\to H^*(H)$ is injective.
In \cite[Corollary $6.7$]{naka} Nakaoka proved that the cohomology of the symmetric group stabilizes. In particular, the result says that if $i<n/2$, then
the map $\res:H^i(S_n)\to H^i(S_{2i})$ is an isomorphism.
We prove a similar result for $\GL_n(\F_q)$.

\begin{theorem} \label{detectq1}
The map $\res: H^i(\GL_m(\F_q))\to H^i(\GL_{n-1}(\F_q))$, where $m\geq n$ and $q\equiv 1\pmod 4$,  is injective for $i<2n-1$. 
\end{theorem}
We have a similar result for the case when $q\equiv 3\pmod 4$ (see Theorem \ref{detectq3}). We also provide criteria for spinoriality of real representations of $\GL_n(\F_q)$ where $q$ is odd (see Theorem \ref{lift}). This recovers the main result in \cite[Theorem $1$]{spjoshi} 

Theorem \ref{detectq1} allows further simplification to calculate the second and fourth Stiefel Whitney classes. Let $\pi$ be a real representation of $\GL_n(\F_q)$, where $q$ is odd. Then we have

	\begin{equation*}
	\begin{split}
		w_2^D(\pi)&=\dfrac{m_{\pi}}{2}\left(\sum_{i=1}^nt_i\right),\quad\text{if}\,\, q\equiv 1\pmod 4, \\ \\
		w_2^{C_2^n}(\pi)&=\dbinom{m_{\pi}}{2}\left(\sum\limits_{i=1}^n v_i^2\right),\quad \text{if}\,\, q\equiv 3\pmod 4,
		\end{split}
	\end{equation*}
	
where $m_{\pi}=\frac{\dim\pi-\chi_{\pi}(h_1)}{2}$.	
Moreover, if $q\equiv 1\pmod 4$, then 
	$$w_4^D(\pi)=\binom{m_{\pi}/2}{2}\sum_{i=1}^nt_i^2+\frac{\dim \pi-\chi_{\pi}(h_2)}{8}\sum_{1\leq i<j\leq n}t_it_j.$$	
As an application of the result we obtain the following property for the real principal series representations (see Section \ref{prin}).
\begin{theorem}\label{psr}
Let $\pi$ be a real principal series representation for $\GL_n(\F_q)$ where $q\equiv 1\pmod 4$. Then $w_4(\pi)=0$ for $n\geq 5$.
\end{theorem}	
We also compute the second SWC for $q$ odd and the fourth SWC for $q\equiv 1\pmod 4$.

We now describe the layout of the paper. In Section \ref{notpre} recall some basic facts on cohomology of the groups $\GL_n(\F_q)$ and $C_m^n$. In Section \ref{vn2} we compute the SWC for $S_n$ invariant real representations of $C_2^n$. We prove the main theorem in Section \ref{pmt}. The theorems \ref{detectq1} and \ref{detectq3} are proved in Section \ref{secdetect}. We compute the second and the fourth SWC in Section \ref{cal} using the previous results. Section \ref{appl} shows some applications of the theory. Section \ref{cr} deals with the cases of $\GL_n(\C)$ and $\GL_n(\R)$. We work out some examples in Section \ref{exam}.

\bigskip

{\bf Acknowledgements:} The authors would like to thank Steven Spallone for helpful conversations. The first author of this paper was supported by a post-doctoral fellowship from Chennai Mathematical Institute (CMI). The second author of this paper was supported by a post-doctoral fellowship
of National Board of Higher Mathematics India (NBHM).

\section{Notation and Preliminaries}\label{notpre}

A complex representation $(\pi,V)$ of a finite group $G$ is called orthogonal if it preserves a non degenerate symmetric bilinear form. We call $\pi$ real if the representation space $V$ is real. For a complex representation $\pi$ one can consider its realisation $(\pi_{\R},V_{\R})$ by forgetting the complex structure of $V$ and regarding it as a real vector space.

\subsection{Irreducible Representations of $C_m^n$}

Take the set of ordered pairs
$$\mc{P}=\{\tau=(\tau_1,\ldots,\tau_n)\mid \tau_j \in C_m, 1\leq j\leq n\}.$$
Consider the subset $\mc{P}'$ of $\mc{P}$ given by
$$\mc{P}'=\{\tau=(\tau_1,\ldots,\tau_n)\mid \tau_j \in \{0,m/2\}, 1\leq j\leq n\}.$$
Consider the linear character 
$\chi_{\tau} = (\chi^{\tau_1} \otimes \chi^{\tau_2}\otimes \cdots \otimes \chi^{\tau_n} )$ of $C_m^n$ (See Section \ref{notpre})).	
We have $\chi_{-\tau}=\chi_{\tau}^{\vee}$.
Note that $\chi_{\tau}$ is real if $\tau\in \mc{P}'$.  Define an equivalence relation on $\mc{P}$ as 
$$\tau\sim (-\tau_1, -\tau_2,\ldots, -\tau_n).$$
Put $\widehat{\mc{P}}=(\mc{P}-\mc{P}')/\sim$. Therefore a real irreducible representations of $C_m^n$ is either of the form $(\chi_{\tau})_{\R}$ for $\tau\in \widehat{\mc{P}}$ or of the form $\chi_{\tau}$ where $\tau\in \mc{P}'$.

\begin{comment}
Consider the linear character $\chi^j$ of $C_m$ given by $\chi^{j}(a)=e^{2\pi ij/m}$ where $a$ is the generator of $C_m$.

For a binary string $a\in \X$ consider the representation $\sgn_a$ of $C_2^n$ defined as $sgn_a = \bigotimes\limits_{j=1}^n (\sgn)^{a_j},$ where
\[
(\sgn)^{\epsilon}=
\begin{cases}
	\mathbb{1},\quad \text{if $\epsilon=0$},\\
	\sgn,\quad \text{if $\epsilon=1$}.
\end{cases}
\]
Consider the representations $\sigma_i=\bigoplus\limits_{l\in \X_i}\sgn_l$ of $C_2^n$ for $1\leq i\leq n$.
\end{comment}

\subsection{Cohomology of $GL_n(\F_q)$, for $q\equiv 1\pmod 4$}

The cohomology ring of $\GL_n(\F_q)$, where $q\equiv 1\pmod 4$ is given by
\begin{equation}\label{cohomGLn}
H^*(\GL_n(\F_q),\ztwo)=\ztwo[a_2,a_4,\ldots,a_{2n}]\otimes E[b_1,b_3,\ldots,b_{2n-1}],
\end{equation} 
where 
$$a_{2l}=\sum_{i_1<i_2<\ldots<i_l}t_{i_1}\cdots t_{i_l},\quad b_{2l-1}=\sum_{i_1<i_2<\ldots<i_l, 1\leq j\leq l}t_{i_1}\cdots \hat{t}_{i_j}\cdots t_{i_l}s_{i_j},\quad 1\leq l\leq n,$$
and $E[b_1,\ldots,b_{2n-1}]$ denotes the alternating algebra. 
For details we refer the reader to \cite[Theorem $8.1$, page $288$]{priddy}.

We know that the symmetric group $S_n$ is the Weyl group of $\GL_n(\F_q)$ with respect to the diagonal torus $D$. It acts on $D$ by permuting the diagonal entries.
\begin{prop}\label{Sninv}Let $\pi$ be a representation of $\GL_n(\F_q)$. Then $\pi\mid_D$ is $S_n$ invariant.
\end{prop}
\begin{proof}
We leave the proof to the reader.
\end{proof}
This enables us to work with $S_n$ invariant representations of $D$. 
\begin{defn}
Let $G'$ be a subgroup of a finite group $G$. 
%Write $\iota: G' \to G$ for the inclusion. 
We say that $G'$ \emph{detects} the (mod $2$) cohomology of $G$, provided that the restriction map
\begin{equation}
\res: H^*(G, \ztwo) \to H^*(G',\ztwo)
\end{equation}
is injective.
\end{defn}

The action of $S_n$ on $D$ induces an action of $S_n$ on $H^*(D)$. We write $H^*(D)^{S_n}$ to denote the subalgebra fixed under the action.
\begin{theorem}[Corollary of Theorem $3$, page $571$ Quillen]\label{Quillen}
The diagonal subgroup $D$ detects the cohomology of $\GL_n(\F_q)$. In fact, the homomorphism 
$$H^*(\GL_n(\F_q))\to H^*(D)$$
induced from the restriction of $\GL_n(\F_q)$ to $D$ is injective. Moreover, the image lies inside $H^*(D)^{S_n}$.
\end{theorem}

\begin{comment}
The next proposition gives a detection result for product of finite groups.

\begin{prop}\label{detectprod} 
 Let the subgroups $H_1$ and $H_2$ detect the mod $2$ cohomology of the finite groups $G_1$ and $G_2$ respectively. Then $H_1\times H_2$ detects the mod $2$ cohomology of $G_1\times G_2$.
 \end{prop}
 
\begin{proof}
%Since $\ztwo$ is a field, from Kunneth formula we have
%$$H^i(G_1\times G_2)\cong \bigoplus_{j+k=i}H^j(G_1)\otimes H^k(G_2).$$
%For each such pair $(j,k)$ there exist injective maps $i_{j}:H^j(G_1)\to H^j(H_1)$ and $i_{k}:H^k(G_2)\to H^k(H_2)$. Using these maps one obtains an injective map $i:H^i(G_1\times G_2)\to H^i(H_1\times H_2)$.
The proof follows from Kunneth formula.
\end{proof}
\end{comment}
 
 \section{SWC of $S_n$ Invariant Representations of $C_2^n$}\label{vn2}

An $S_n$ invariant representation of $C_2^n$ will be of the form
\begin{equation}\label{piv2n}
\pi=\bigoplus_{i=0}^{n}c_i\sigma_i,
\end{equation}
where $\sigma_i$ is as in \eqref{sig}.

\begin{prop} \label{chi_i,k}
The character value $\chi_{\sigma_i}(h_k)$ is equal to the coefficient of $y^i$ in the expression $(1-y)^k(1+y)^{n-k}$.
\end{prop}

\begin{proof}
Fix two integers $i,k$ such that $0\leq i,k\leq n$.
Let $R_l$ denote the set of elements $a\in \X_i$ such that the first $k$ coordinates of $a$ contains $l$ number of $1$'s. Then for $a\in R_l$ we have $\chi_{\sgn_a}(h_k)=(-1)^l.$ Put $\tau_{l} = \bigoplus\limits_{a \in R_l} \sgn_ a$. Then 
$$\chi_{\tau_{l}}(h_k)=\binom{k}{l}\binom{n-k}{i-l}(-1)^l.$$
Observe that $\bigoplus\limits_{l=0}^k\tau_{l}=\sigma_i$.
Thus
\begin{equation}
\chi_{\sigma_i}(h_k)=\sum_{l=0}^{k}\binom{k}{l}\cdot \binom{n-k}{i-l}(-1)^l.
\end{equation}
Note that this gives the coefficient of $y^i$ in $(1-y)^k(1+y)^{n-k}$.

%Consider the product
%$$\left(\sum_{l=0}^{k}(-1)^l\binom{k}{l}y^l\right)\cdot\left(\sum_{l=0}^{i}\binom{n-k}{i-l}y^{-l}\right).$$
%Then $\chi_{\sigma_i}(h_k)$ is the constant term in the above product.
%Observe that it is also constant term in the product 
%$$\left(\sum_{l=0}^{x}(-1)^l\binom{k}{l}y^l\right) \cdot \left(\sum_{l=i-(n-k)}^i\binom{n-k}{i-l}y^{-l}\right), $$
%%as we are adding some positive powers of $y$ in the second bracket. 
%We make a change of variable $z=i-l$. 
%The above expression is
%\begin{align*}
%&y^{-i}\left(\sum_{l=0}^{k}(-1)^l\binom{k}{l}y^l\right) \cdot \left(\sum_{i-l=(n-k)}^{0}\binom{n-k}{i-l}y^{i-l}\right)\\
%&=y^{-i}(1-y)^k\left(\sum_{x=0}^{n-k}\binom{n-k}{z}y^{z}\right)\\
%&= y^{-i}(1-y)^k(1+y)^{n-k}.
%\end{align*}
%This gives the desired result.

\end{proof}
%We have 
%$$\sum_{l=0}^{\text{min}\{i,k\}}\binom{k}{l}\cdot \binom{n-k}{i-l}(-1)^l=\text{constant term in the expression}\,\,y^{-i}(1-y)^k(1+y)^{n-k}.$$

%\begin{prop}\label{M}
%	Consider the $(n+1) \times (n+1)$ matrix  $M = [m_{ki}]$ such that $m_{ki} = \chi_{\sigma_i}(e_k)$ where $0 \leq i,k \leq n$  is integral and invertible with the property that its inverse is also integral.
%\end{prop}
%
%\begin{proof}
%Consider the $(n+1) \times (n+1)$ matrices  $A=[a_{kl}]$, and $B=[b_{li}]$, where $0\leq k,l,i\leq n$, such that  $a_{kl}=(-1)^l\binom{k}{l}$ and $b_{li}=\binom{n-k}{i-l}$. It is easy to see that $M = A\cdot B$. Observe that $A$ is a lower triangular matrix with nonzero diagonal entries $a_{kk}=(-1)^k$. Hence its determinant is $\pm 1$. On the  hand $B$ is an upper triangular matrix with nonzero diagonal entries $b_{ll}=1$. Hence its determinant is $1$. In  particular, $M=A\cdot B$ is invertible with determinant $\pm 1$ and hence its inverse is also integral.
%\end{proof}

From \eqref{piv2n} we have
\begin{equation}\label{}
	 \chi_{\pi}(h_i) = \sum_{j=0}^n c_j \cdot \chi_{\sigma_j}(h_i)
\end{equation}
    
Define matrix $M =(m_{ij})$ where $0 \leq i,j \leq n$ such that 
\begin{equation}\label{mijchisigmahi}
	m_{ij} = \chi_{\sigma_j}(h_i).
\end{equation}

Take  
%$a=\begin{pmatrix}
%\chi_{\pi}(e_0)\\
%\chi_{\pi}(e_1)\\
%\vdots\\
%\chi_{\pi}(e_n)
%\end{pmatrix}$ and $m= \begin{pmatrix}
%m_0\\
%m_1\\
%\vdots\\
%m_n
%\end{pmatrix}$.
$a={}^t\left(\chi_{\pi}(h_0),\ldots, \chi_{\pi}(h_n)\right)\quad \text{and}\quad m={}^t\left(c_0,\ldots, c_n\right),
$ where ${}^tA$ denotes the transpose of the matrix $A$.  Observe that 
\begin{equation}\label{amM}
a= M\cdot m.
\end{equation}
%By proposition \ref{M} $M^{-1}$ is integral. Thus if $a$ is integral matrix with all even entries then so is $m$.

%\begin{lemma}\label{sw1}
%	If all the entries $m_i$ in $m$ are even then the total Stiefel-Whitney of $\pi$ class is $1$.
%\end{lemma}
%\begin{proof}
%	The total Stiefel-Whitney class of $\pi$ is 
%	\begin{align}
%	w(\pi) =& \prod_{1\leq i \leq n} (1+s_i)^{m_1} \cdot \prod_{1 \leq i < j \leq n}(1+s_i + s_j)^{m_2} \cdots\\
%	=&\prod_{1\leq i \leq n} ((1+s_i)^2)^{m_1/2} \cdot \prod_{1 \leq i < j \leq n}((1+s_i + s_j)^2)^{m_2/2} \cdots
%	\end{align}
%	Since $s_i^2=0 \forall i$, and $H^*(C_2^n, \ztwo)$ is a vector space  with field $\ztwo$ we have $(s_{i_i} + s_{i_2} + \cdots + s_{i_k})^2=0$. Hence the claim follows.  
%	
%\end{proof}

\begin{theorem}\label{thm1}
	We have 
	$$M^2=2^nI_{n+1},$$ 
	where $I_{n+1}$ denotes the identity matrix of rank $n+1$.
\end{theorem}

\begin{proof}
	Let 
	$$R_i=(m_{i0},\ldots, m_{in}),$$
	be the $i$-th row of $M$. Note that from Proposition \ref{chi_i,k} the element $m_{il}$, $0\leq l\leq n$, is equal to the coefficient of $y^l$ in $(1-y)^i(1+y)^{n-i}$.
	
	We calculate the $x$-th row in the matrix $M^2$.
	\begin{align*}
	\sum_{l=0}^n(M^2)_{xl}y^l&=\sum_{l=0}^n\left(\sum_{i=0}^nm_{xi}m_{il}y^l\right)\\
	&=\sum_{i=0}^nm_{xi}\left(\sum_{l=0}^nm_{il}y^l\right)\\
	&=\sum_{i=0}^{n}\chi_{\sigma_i}(h_x)(1-y)^i(1+y)^{n-i}\\
	&=\sum_{i=0}^{n}\left(\sum_{l=0}^{x}\binom{x}{l}\cdot \binom{n-x}{i-l}(-1)^l\right)(1-y)^i(1+y)^{n-i}\\
	&=\sum_{l=0}^{x}\binom{x}{l}(-1)^l\left(\sum_{i=0}^{n}\binom{n-x}{i-l}(1-y)^i(1+y)^{n-i}\right)\\
	&=\sum_{l=0}^{x}\binom{x}{l}(-1)^l\left(\sum_{t=0}^{n-x}\binom{n-x}{t}(1-y)^{t+l}(1+y)^{n-t-l}\right)\\
	&=\sum_{l=0}^{x}\binom{x}{l}(-1)^l(1-y)^l\left(\sum_{t=0}^{n-x}\binom{n-x}{t}(1-y)^{t}(1+y)^{n-t-l-x+x}\right)\\
	&=\sum_{l=0}^{x}\binom{x}{l}(-1)^l(1-y)^l(1+y)^{x-l}\left(\sum_{t=0}^{n-x}\binom{n-x}{t}(1-y)^{t}(1+y)^{n-x-t}\right)\\
	&=\sum_{l=0}^{x}\binom{x}{l}(y-1)^l(y+1)^{x-l}(1-y+1+y)^{n-x}\\
	&=(2y)^x2^{n-x}\\
	&=2^ny^x.
	\end{align*}
\end{proof}

Since $M$ is invertible, we have 
\begin{equation}\label{m=Minva}
m = M^{-1} \cdot a.
\end{equation}

\begin{prop}\label{pisigmak}
Let $\pi$ be the representation of $C_2^n$ as in equation \eqref{piv2n}. Then we have
$$c_k=\frac{1}{2^n}\sum_{i=0}^{n}\chi_{\sigma_i}(h_k)\chi_{\pi}(h_i).$$
%where $\sigma_k=\bigoplus\limits_{l\in\X_k}\sgn_l$. 
\end{prop}

\begin{proof}
The proof follows from  Equations \eqref{mijchisigmahi} \eqref{m=Minva} and Theorem \ref{thm1}.
\end{proof}

%Let $m =[r_0, r_1, \ldots, r_n]^t $ and $J= [0,\binom{n-1}{0},\binom{n-1}{1}, \ldots, \binom{n-1}{n-1}]$.
%So we have $Jm= \sum\limits_{j=1}^n \binom{n-1}{j-1} r_j$. Now we have $m = (1/2^n)Ma$ from Theorem \ref{thm1} and \eqref{m=Minva}. As in the proof of Theorem 1, let $R_i = (m_{i0},m_{i1}, \ldots , m_{in})$ be the $i$-th row of $M$ which satisfies $\sum_{j=0}^n m_{ij}y^j= (1-y)^i(1+y)^{n-i}$. We calculate $JM$.
%%
%\begin{lemma}\label{jm}
%We have $Jm=\dfrac{\chi_{\pi}(h_0)-\chi_{\pi}(h_1)}{2}$.
%\end{lemma}
%
%\begin{proof}
%\begin{align*}
%JM&= \sum_{i=1}^{n}\binom{n-1}{i-1}(1-y)^i(1+y)^{n-i} \quad \text{Put} \quad i-1= \alpha\\
%&=\sum_{\alpha=0}^{n-1}\binom{n-1}{\alpha}(1-y)^{\alpha+1}(1+y)^{n-1-\alpha}\\
%&=(1-y)\sum_{\alpha=0}^{n-1}\binom{n-1}{\alpha}(1-y)^{\alpha}(1+y)^{n-1-\alpha}\\
%&=(1-y)(1-y +1+y)^{n-1}\\
%&=2^{n-1}(1-y)\\
%&=[2^{n-1}, -2^{n-1}, 0 , \ldots, 0] 
%\end{align*}
%
%
%Thus we have 
%
%\begin{align*}
%Jm&= (1/2^n)JMa\\
%&= (1/2^n)[2^{n-1}, -2^{n-1}, 0, \ldots , 0 ]{}^t[\chi_{\pi}(e_0),\chi_{\pi}(e_1), \ldots , \chi_{\pi}(e_n)]\\
%&= (1/2)(\chi_{\pi}(h_0) - \chi_{\pi}(h_1))\\
%\end{align*}
%\end{proof}

\begin{theorem}\label{wpiv2n}
Let $\pi$ be an $S_n$ invariant representation of $C_2^n$ of the form $\pi=\bigoplus\limits_{i=1}^{n}c_i\sigma_i$.Then we have 
\begin{equation}
		w(\pi)=\prod_{i=1}^{n}\left(\prod_{l\in\X_i}(1+l\cdot v)\right)^{c_i}
	\end{equation}
where $c_i=\frac{1}{2^n}\sum\limits_{k=0}^{n}\chi_{\sigma_k}(h_i)\chi_{\pi}(h_k)$.
\end{theorem}

\begin{proof}
	We have $$w(\sigma_i) = \prod_{l\in\X_i}(1+l\cdot v)$$ and $$w(\pi)=\prod_{i=1}^n(w(\sigma_i))^{c_i}.$$ Now from Proposition \ref{pisigmak}  we have 
	$$c_i =\frac{1}{2^n}\sum_{k=0}^{n}\chi_{\sigma_k}(h_i)\chi_{\pi}(h_k). $$ 
Hence the result follows.
	
	\end{proof}

%\section{SWC of $S_n$ Invariant Representation of $C_m^n$  }\label{vnm}	

%Consider the additive cyclic group $C_n$ such that $n\equiv 0\pmod 4$. Let $\chi_j$ denote the character of $C_n$ given by $\chi_j(1) = \zeta_n^j$, where $\zeta_n = e^{2\pi i/n}$. 
%We write $\sgn$ for $\chi_{n/2}$. For a complex representation $\pi$, we denote by $\pi_{\R}$ the realization of $\pi$. Moreover we write $S(\pi)$ for $\pi \oplus \pi^{\vee}$, where $\pi^{\vee}$ is the dual representation of $\pi$.
%The group under consideration is $\GL_n(\F_q)$ and put $n=q-1$.	

\begin{comment}
	
The symmetric group $S_n$ acts on $C_m^n$ by permuting the components.
Let $\phi$ be a real $S_n$ invariant representation of $C_m^n$.  Then we may write
\begin{equation}\label{rep1} 
	\phi= m_0 \mathbb{1} \bigoplus_{i=1}^n m_i \sigma_i \bigoplus_{\tau \in \widehat{\mc{P}}}m_{\tau} (\chi_{\tau})_{\R}.
\end{equation}
%where $\sigma_i = \bigoplus\limits_{a \in \X_i} \sgn_a $.

%We have $$H^*(C_m^n, \ztwo) = \dfrac{\F_2[s_1,\cdots s_n, t_1, \cdots, t_n]}{(s_i^2=0)},$$
%where $s_i=w_1(\sgn_a)$ where $a_l=\delta_{il}$. Also $t_j=w_2(S(\chi_{\tau}))$ for $\tau_l=\delta_{jl}$.

Thus we have
\begin{equation}\label{wpi}
	w(\phi)= w(m_0 \mathbb{1} \bigoplus_{i=1}^n m_i \sigma_i)w(\bigoplus_{\tau \in \widehat{\mc{P}}} m_\tau ((\chi_{\tau})_{\R})
\end{equation}
\end{comment}

\section{proof of main theorem}\label{pmt}

For a subset $\I =\{i_1, i_2, \ldots,i_l\} $ of $\{1,2,\ldots, n\}$ and for a multiset  $\J = \{j_1,j_2, \ldots, j_k\}$ where $j_i \in \{1,2, \ldots, n\}$ we define $$ s^{\I} = s_{i_1} \cdots s_{i_l} \text{ and } t^{\J}= t_{j_1} t_{j_2} \cdots t_{j_k}.$$
Observe that $H^*(C_{4m}^n) = \dfrac{\ztwo[s_1, \ldots, s_n, t_1, \ldots, t_n]}{(s_1^2,\ldots, s_n^2)}$ is a $\ztwo$ vector space with basis $s^{\I}t^{\J}$, where $\I \,\,( \text{ resp. } \J)$ varies over  all subsets ( resp. multi-sets) over $\{1,2, \ldots ,n\}$.

\begin{prop}\label{polysi}
	Let $q\equiv 1 \pmod 4$. If $\pi$ is a real representation of $\GL_n(\F_q)$ and if $w(\pi) = \sum\limits_{\I, \J} a_{\I, \J} s^{\I} t^{\J}$, then elements of the form $s^{\I}$ have coefficient $0$ i.e. $a_{\I, \emptyset} = 0$ if $|\I|\geq 2$.
\end{prop}
\begin{proof}
	Consider the projection map $$\mathrm{pr}:\dfrac{\ztwo[s_1,\ldots,s_n,t_1,\ldots,t_n]}{(s_1^2,\ldots, s_n^2)}\to \dfrac{\ztwo[s_1,\ldots,s_n]}{(s_1^2,\ldots, s_n^2)}$$
	  defined by $\mathrm{pr}(s_i)=s_i$ and $\mathrm{pr}(t_i)=0$ for all $i$. This gives $\pr(a_{2i})=0$ for all $i$ and $\pr(b_{2i-1})=0$ for $i>1$. Then from Equation \eqref{cohomGLn} we have 
	\begin{equation}\label{impr}   
		\mathrm{pr}\left(H^*(\GL_n(\F_q))\right)=\ztwo[b_1].
	\end{equation}
Now if $s^I$ appears as a term in $w(\pi)$, where $|I|\geq 2$, then it survives in $\pr(w(\pi))$. But this contradicts Equation \eqref{impr}. 

\begin{comment}
 Now since $w(\pi) \in H^*(\GL_n(\F_q))$, suppose a term of the form $s^{\I}$, where $|\I|>1$, has coefficient $1$ in $w(\pi)$, then it survives in image $\mathrm{pr}(w(\pi))$, which contradicts Equation $\eqref{impr}$. 
\end{comment} 
\end{proof}

%\begin{theorem}
%	Let $\pi$ be a real representation of $\GL_n(\F_q)$ with $q \equiv 1 \pmod 4$. Then the total Stiefel Whitney class is 
%	$$w^D(\pi) = (1+\delta b_1)\prod_{i=1}^n(\prod_{l \in \X_i}(1 + l\cdot t))^{c_i/2},$$
%	where $l\cdot t = \sum_j l_j t_j$.
%\end{theorem} 

\begin{proof}[\textbf{Proof of Main Theorem} \ref{main}] 
	Let  $\widehat{P}$ be as defined in Section \ref{notpre}. Let $\pi$ be a real representation of $\GL_n(\F_q)$. Recall that for a subgroup $H$ of $\GL_n(\F_q)$ we write $w_i^H(\pi)$ and $w^H(\pi)$ to denote $w_i(\pi\mid_{H})$ and $w(\pi\mid_{H})$ respectively.
	
	Since $\pi \mid_D$ is an $S_n$ invariant representation of $D$ we write
	
	\begin{equation}
		\pi \mid_D = \mathbbm{1}^{\oplus m_0} \bigoplus_{i=1}^n \sigma_i^{\oplus m_i} \bigoplus_{\tau \in \widehat{P}} ((\chi_\tau)_\R)^{ \oplus m_\tau}.
	\end{equation}
	Therefore we may write
	\begin{align*}
		w^D(\pi) =& w(\mathbbm{1}^{\oplus m_0} \bigoplus_{i=1}^n \sigma_i^{\oplus m_i})\cdot w(\bigoplus_{\tau \in \widehat{P}} ((\chi_\tau)_\R)^{\oplus m_\tau})\\
		=&(1 + Q(s_1, \ldots, s_n))(1+ P(t_1, \ldots, t_n)) \left(\text{ where } P, Q \text{ do not have constant term}\right),\\
		=& (1 + \delta b_1) (1 + P(t_1, \ldots t_n)) \left(\text{ follows from Proposition} \ref{polysi}\right),\\
	\end{align*}
	This calculation shows that $w_1^D(\pi) = \delta b_1$. From this we conclude that \[\delta= \begin{cases} 0,\text{ if } \det \pi =1,\\ 1, \text{ otherwise. }
	\end{cases}
	\]
	Consider the representation $\psi = (\chi^1)_\R$ of $C_{q-1}$. 
	%where $\chi$ maps the generator $g$ of $C_{q-1}$ to the primitive $q-1^{\text{th}}$ root of unity $\zeta_{q-1}$ . 
	Let $C_2 < C_{q-1}$ be the subgroup of order $2$. Then the restriction map $\res: H^*(C_{q-1}) \to H^*(C_2)$ maps $t = w_2(\psi)$ to $w_2^{C_2}(\psi)=w_2(\sgn \oplus \sgn) = v^2$. Thus the restriction map $\res_n:H^*(C_{q-1}^n) \to H^*(C_2^n)$ maps $t_i$ to $v_i^2$ and $s_i$ to $0$.
	% hence the $res_n$ is injective on the subring $\ztwo[t_1,\ldots,t_n]$. Therefore we get
	\begin{align}
		w^{C_2^n}(\pi) &= 1+ P(v_1^2, \ldots, v_n^2)\\
		&=(1 + P(v_1, \ldots, v_n))^2. \label{resc2n} 
	\end{align}
	
	On the other hand $\pi\mid_{C_2^n}$ is an $S_n$ invariant representation, therefore by Theorem \ref{wpiv2n} we have
	
	$$ w^{C_2^n}(\pi)=\prod_{i=1}^n (\prod_{l\in \X_i} (1+ \sum_j l_j v_j))^{c_i},$$
	where $c_i = \dfrac{1}{2^n}\sum\limits_{k=0}^n \chi_{\sigma_k}(h_i)\chi_\pi(h_k).$
	Now since $H^*(C_2^n)= \ztwo[v_1, \ldots,v_n]$ is a Unique Factorization domain, we have
	$1+P(v_1, \ldots , v_n) = \prod_{i=1}^n(\prod_{l \in \X_i}(1 + \sum_j l_j v_j))^{c_i/2}.$
	
	Thus we obtain $$w^D(\pi) = (1+\delta b_1)\prod_{i=1}^n\left(\prod_{l \in \X_i}(1 + l\cdot t)\right)^{c_i/2},$$
	where $l\cdot t = \sum_j l_j t_j$.
	
	In the case $q \equiv 3 \pmod{4}$ we have 
	\begin{comment}
	$H^*(D) = \ztwo[v'_1,v'_2, \ldots, v'_n]$ and let $C_2 < C_{q-1}$ be the unique cyclic group of order $2$. Moreover
\end{comment}	
$$\mathrm{res^D_{C_2^n}}:H^*(D) \to H^*(C_2^n)$$ 
is an isomorphism as $C_2^n$ is the $2$ Sylow subgroup of $D$. (see \cite[Propositions $5$ and $6$]{ganjo}).
%mapping $v'_i$ to $v_i$. Thus $$\mathrm{res}: H^*(\GL_n(\F_q)) \to H^*(D) \to H^*(C_2^n)$$ is injective. 
Now we apply Theorem \ref{wpiv2n} to obtain 
	
	$$w^{C_2^n}(\pi)=\prod_{i=1}^n (\prod_{l\in \X_i} (1+ \sum_j l_j v_j))^{c_i},$$ where again $c_i = \dfrac{1}{2^n}\sum_k \chi_{\sigma_k}(h_i)\chi_\pi(h_k). $
	This completes the proof of the main theorem.
\end{proof} 

\begin{remark}
Since $c_i$ is the multiplicity of $\sigma_i$ in $\pi\mid_{C_2^n}$ , it is an integer, where $\sigma_i$ is a representation of $C_2^n$. Moreover, if $q\equiv 1\pmod 4$, then $\sigma_j\mid_{C_2^n}=\binom{n}{j}\mathbb{1}$, where $\sigma_j$ is a representation of $D$. Also $(\chi)_{\R}\mid_{C_2^n}=2\chi\mid_{C_2^n}$. Therefore the integers $c_i$ are even for $i\geq 1$.
\end{remark}

\begin{remark}
For a real representation of $\GL_n(\F_q)$, where $q\equiv 1\pmod 4$, one can write the expression for $w(\pi)$ in terms of the generators of $H^*(\GL_n(\F_q))$, namely $a_{2l}, b_{2l-1}$ for $1\leq l\leq n$. This follows from the fact that the terms $a_{2l}$ are elementary symmetric functions in $t_i$ for $1\leq i\leq n$. For example if $n=3$, then
\begin{equation*}
	\begin{split}	
		w(\pi)=&\left(1+\delta\cdot \sum_{i=1}^3 s_i\right)\left((1+t_1)(1+t_2)(1+t_3)\right)^{m_1}\\
		&\left((1+t_1+t_2)(1+t_1+t_3)(1+t_2+t_3)\right)^{m_2}\\
		&(1+t_1+t_2+t_3)^{m_3}.
	\end{split}
\end{equation*}
Therefore one has
\begin{equation*}
	\begin{split}	
		w(\pi)=&(1+\delta\cdot b_1)(1+a_2+a_4+a_6)^{m_1}\\
		&(1+a_2^2+2a_2+a_4+a_2a_4-a_6)^{m_2}\\
		&(1+a_2)^{m_3}.
	\end{split}
\end{equation*}
\end{remark}

\section{Detection Results}\label{secdetect}

We write 
$$E[x_1,\ldots,x_n]=\dfrac{\ztwo[x_1,\ldots,x_n]}{(x_1^2,\ldots,x_n^2)}.$$
Let $D_{n}$ denote the diagonal torus of $\GL_n(\F_q)$, where $q\equiv 1\pmod 4$. Take $m\geq n$. We have
$$H^*(D_m)=\dfrac{\ztwo[t_1,\ldots,t_m,s_1,\ldots,s_m]}{( s_1^2,\ldots, s_n^2)},$$
and
$$H^*(D_{n-1})=\dfrac{\ztwo[t_1',\ldots,t_{n-1}',s_1',\ldots,s_{n-1}']}{( {s_1'}^2,\ldots,{s_n'}^2)}.$$
Consider the inclusion map $j_n:\GL_{n-1}(\F_q)\to \GL_m(\F_q)$ given by 
$$j_n(A)=\begin{pmatrix}
A & 0\\
0 & I_{m-n+1}
\end{pmatrix},$$
where $A\in \GL_{n-1}(\F_q)$. It will induce a map on the diagonal subgroups namely $j_n:D_{n-1}\to D_m$ (we use the same notation for the induced map). We have 
\[
j_n^*(t_l)=
\begin{cases}
t_l',& \text{if $l<n$}\\
0, & \text{if $l\geq n$}
\end{cases}
\]
Similarly
\[
j_n^*(s_l)=
\begin{cases}
s_l',& \text{if $l<n$}\\
0, & \text{if $l\geq n$}
\end{cases}
\]
It will induce a map 
$$j_n^*:\ztwo[a_2,\ldots,a_{2m}]\otimes E[b_1,b_3, \ldots , b_{2m-1}]\to \ztwo[a_2',\ldots,a_{2n-2}'] \otimes E[b_1',\ldots,b_{2n-3}'] ,$$
where $$j_n^*(a_{2l})= a_{2l}'=\sum\limits_{1\leq i_1<i_2<\ldots<i_l\leq n-1}t_{i_1}'\cdots t_{i_l}',$$ 
for $l<n$ and 
$$j_n^*(b_{2k-1}) = b_{2k-1}' = \sum\limits_{1 \leq j \leq k} \left(\sum\limits_{1\leq i_1<i_2<\ldots<i_k\leq n-1} t_{i_1}'\cdots \hat{t'}_{i_j}\cdots t_{i_k}' s_{i_j}'\right),$$
for $k<n$.
Therefore $j_n^*(a_{2l})=0$ when $l\geq n$ and $j_n^*(b_{2k-1})=0$  for $k\geq n$.
%\begin{theorem} \label{detectq1}
%The map $\res: H^i(\GL_m(\F_q))\to H^i(\GL_{n-1}(\F_q))$, where $m\geq n$ and $q\equiv 1\pmod 4$ is injective for $i<2n-1$. 
%\end{theorem}

\begin{proof}[Proof of Theorem \ref{detectq1}]
Put $R_m =\ztwo[a_2,\ldots,a_{2m}] $, $E_m= E[b_1,b_3, \ldots , b_{2m-1}], R'_{n-1}=\ztwo[a_2',\ldots,a'_{2n-2}]$ and $E_{n-1}'= E[b_1',b_3', \ldots , b'_{2n-3}]$. Since the elements $a_2,\ldots, a_{2n-2}$ are algebraically independent, we have 
$$\ker j_n^*\mid _{R_m}=\langle a_{2n}, a_{2n+2},\ldots,a_{2m} \rangle.$$
Take $R_m=\bigoplus_{i}A_i$ and $R_{n-1}'=\bigoplus_{i}A_i'$, where $A_i$ and $A_i'$ are the $i$-th degree components. Then comparing degrees we have $A_i\cap \ker j_n^*=0$, for $i<2n$. So the map $j_n^*:A_i\to A_i'$ is injective for $i<2n-1$.

Note that $E[b_1,b_3,\ldots,b_{2m-1}]$ is a $\ztwo$ vector space
with basis elements 
$$b_{2i_1-1}\cdots b_{2i_k-1},$$ 
where $1\leq i_1<i_2<\cdots<i_k\leq m$ and $1\leq k\leq m$. In each basis element a particular $b_j$ can have exponent $1$ or $0$. Therefore $\dim E[b_1,b_3,\ldots,b_{2m-1}]=2^m$. Similarly $\dim E[b_1',b_3',\ldots,b_{2n-3}']=2^{n-1}$. Since $j_n^*$ is surjective we have $\dim \ker j_n^*\mid_{E_m}=2^m-2^{n-1}$. 

Consider the ideal $\mathfrak{I}=\langle b_{2n-1}, b_{2n+1},\ldots, b_{2m-1}\rangle$ of $E_m$. As a $\ztwo$ vector space $\dim \mathfrak{I}=2^m-2^{n-1}$. Also for $x\in \mathfrak{I}$ one has $j_n^*(x)=0$. Therefore we obtain  
$$\ker j_n^*\mid_{E_m}=\langle b_{2n-1}, b_{2n+1},\ldots, b_{2m-1}\rangle.$$
Let $E_m = \bigoplus_i B_i$ and $E_{n-1}'=\bigoplus_jB_j'$. Comparing the degrees we get $\ker j_n^*\cap B_i=0$ for $i<2n-1$. So the map $j_n^*:B_i\to B_i'$ is injective for $i<2n-1$.

The $i$th graded part of $R_m\otimes E_m$ is of the form $\bigoplus\limits_{k+l=i}A_k\otimes B_l$.
Similarly, the $i$th graded part of $R_{(n-1)}'\otimes E_{(n-1)}'$ is of the form $\bigoplus\limits_{k+l=i}A_k'\otimes B_l'$. To show the injectivity of the map $j_n^*$
%$$j_n^*\otimes j_n^*:\ztwo[a_2,\ldots,a_{2m}]\otimes E[b_1,b_3,\ldots, b_{2m-1}]\to \ztwo[a_2',\ldots,a_{2n-2}']\otimes E[b_1',b_3',\ldots, b_{2n-3}']$$
it is enough to show that the map $j_n^*:A_k\otimes B_{i-k}\to A_k'\otimes B_{i-k}'$ is injective for every $0\leq k\leq i<2n-1$. It follows from the fact that the maps $j_n^*:A_k\to A_k'$ and $j_n^*:B_k\to B_k'$ are injective for $0\leq i<2n-1$.
 %Consider the bases $\{e_i\}$ and $\{f_j\}$ of $A_k$ and $B_{i-k}$ respectively. Since $j_n^*$ is injective the sets $\{j_n^*(e_i)\}$ and $\{j_n^*(f_j)\}$ are linearly independent and they can be extended to the bases of $A_k'$ and $B_{i-k}'$ respectively. So the set $j_n^*(e_i)\otimes j_n^*(f_j)$ is linearly independent in $A_k'\otimes B_{i-k}'$. 
%Therefore the map $j_n^*:R_k\otimes E_{i-k}\to R_k'\otimes E_{i-k}'$ is injective.
\end{proof}

\begin{theorem} \label{detectq3}
The map $\res:H^i(\GL_m(\F_q))\to H^i(\GL_{n-1}(\F_q))$, where $m\geq n$ and $q\equiv 3\pmod 4$ is injective for $i<n$.
\end{theorem}

\begin{proof}
	
From Theorem \ref{Quillen} we obtain that the map $\res_m:H^i(\GL_m(\F_q))\to H^i(D_m)^{S_m}$ is injective.	
Let $\alpha:H^i(\GL_m(\F_q))\to H^i(\GL_{n-1}(\F_q))$ denote the restriction map.
	
	We have $H^*(D_k)^{S_k}=\ztwo[v_1,\ldots,v_k]^{S_k}$. Consider the map 
	$$\beta_n^*: \ztwo[v_1,\ldots,v_m]^{S_m}\to \ztwo[v_1',\ldots,v_{n-1}']^{S_{n-1}},$$
	given by 
\[
\beta_{n}^*(v_i)=
\begin{cases}
v_i',&\text{if $i<n$},\\
0, &\text{if $i\geq n$}.
\end{cases}
\]	
Note that $\beta_n^*$ induces a map 
$$\beta_{ni}^*:H^i(D_m)^{S_m}\to H^i(D_{n-1})^{S_{n-1}}.$$
Consider the following commutative diagram.
\begin{center}
	\begin{tikzpicture}
		\node (A1) at (0,2) {$H^i(\GL_m(\F_q)) $};
		\node (A2) at (4,2) {$H^i(\GL_{n-1}(\F_q))$};
		\node (B1) at (0,0) {$H^i(D_m)^{S_m}$};
		\node (B2) at (4,0) {$H^i(D_{n-1})^{S_{n-1}}$};
		
		%\path[->,font=\scriptsize,>=angle 90]
		\draw[->](A1)to node [above]{$\alpha$} (A2);
		\draw[->](A1) to node [right]{$\res_m$} (B1);
		\draw[->](B1) to node [above]{$\beta_{ni}^*$} (B2);
		\draw[->](A2) to node [right]{$\res_{n-1}$} (B2);
	\end{tikzpicture}
\end{center}
Since the maps $\res_m$ and $\res_{n-1}$ are injective, it is enough to prove that the map $\beta_{ni}^*$ is injective. Let $f_i^k$ denote the elementary symmetric polynomial of degree $i$ in $k$ variables $v_1,\ldots, v_k$, where $k\geq i$.  Note that $\beta_{ni}^*(f_i^k)\neq 0$ for $1\leq i\leq n-1\leq k$. We know that the set of elementary symmetric functions form a basis for the algebra $\ztwo[v_1,\ldots,v_m]^{S_m}$. An element in $\ztwo[v_1,\ldots,v_m]^{S_m}$ of degree $i\leq n-1$ is of the form $P(f_1^m,\ldots,f_i^m)$, where $P$ is a polynomial (see \cite[Theorem $6.1$, page $191$]{lang}). We have $\beta_{ni}^*(P(f_1^m,\ldots,f_i^m))=P(f_1^{n-1},\ldots,f_i^{n-1})$. Therefore $\beta_n^*(P(f_1^m,\ldots,f_i^m))\neq 0$ as the elements $f_1^{n-1},\ldots,f_i^{n-1}$ are algebraically independent. So $\beta_{ni}^*$ is injective for $i\leq n-1$. 
\end{proof}

\section{Calculations of Some SWC}\label{cal}

\subsection{Calculations Using the Main Theorem}
Here we calculate small SWCs using the main result. 

\subsubsection{For $\mathbf{q\equiv 1\pmod 4}$}
From the main theorem we have for $q\equiv 1\pmod 4$,

$$w(\pi) = (1+\delta b_1) \prod_{i=1}^n\left(\prod_{p \in \X_i}(1 + p\cdot t) \right)^{c_{i}/2},$$

Here we calculate $w_2(\pi)$.
It is clear from the above expression that
\begin{align*}
	w_2(\pi) &= \sum\limits_{j=1}^n (c_j/2)\left(\sum\limits_{1 \leq  i_1 < i_2 < \cdots < i_j \leq n}(t_{i_1} + t_{i_2}+ \cdots + t_{i_j})\right)\\
	&=\frac{1}{2} \sum\limits_{j=1}^n c_j\binom{n-1}{j-1}\left(\sum\limits_{i=1}^n t_i\right)\\
	&=\frac{1}{2}\sum\limits_{j=1}^n c_j\binom{n-1}{j-1}a_2.
\end{align*}

Let $m ={}^t[c_0, c_1, \ldots, c_n] $ and $J= [0,\binom{n-1}{0},\binom{n-1}{1}, \ldots, \binom{n-1}{n-1}]$.
So we have 
\begin{equation} \label{eqjm} 
	Jm= \sum\limits_{j=1}^n \binom{n-1}{j-1} c_j.
\end{equation} 
Using Equation \eqref{amM} and Theorem \ref{thm1} we have $m = (1/2^n)Ma$. As in the proof of Theorem 1, let $R_i = (m_{i0},m_{i1}, \ldots , m_{in})$ be the $i$-th row of $M$ which satisfies $\sum\limits_{j=0}^n m_{ij}y^j= (1-y)^i(1+y)^{n-i}$. We calculate $JM$.

\begin{prop}\label{jm}
	We have $Jm=\dfrac{\chi_{\pi}(h_0)-\chi_{\pi}(h_1)}{2}$.
\end{prop}

\begin{proof}
	
	We identify a $1\times (n+1)$ vector $(m_{i0},\ldots,m_{in})$ with the polynomial $\sum\limits_{j=0}^n m_{ij}y^j= (1-y)^i(1+y)^{n-i}$. 	
	We compute
	\begin{align*}
		JM&= \sum_{i=1}^{n}\binom{n-1}{i-1}(1-y)^i(1+y)^{n-i} \quad \text{Put} \quad i-1= \alpha\\
		&=\sum_{\alpha=0}^{n-1}\binom{n-1}{\alpha}(1-y)^{\alpha+1}(1+y)^{n-1-\alpha}\\
		&=(1-y)\sum_{\alpha=0}^{n-1}\binom{n-1}{\alpha}(1-y)^{\alpha}(1+y)^{n-1-\alpha}\\
		&=(1-y)(1-y +1+y)^{n-1}\\
		&=2^{n-1}(1-y)\\
		&=[2^{n-1}, -2^{n-1}, 0 , \ldots, 0]. 
	\end{align*}
	Thus we have 
	\begin{align*}
		Jm&= (1/2^n)JMa\\
		&= (1/2^n)[2^{n-1}, -2^{n-1}, 0, \ldots , 0 ]{}^t[\chi_{\pi}(h_0),\chi_{\pi}(h_1), \ldots , \chi_{\pi}(h_n)]\\
		&= (1/2)(\chi_{\pi}(h_0) - \chi_{\pi}(h_1)).\\
	\end{align*}
\end{proof}

So for $q\equiv 1\pmod 4$ we have $$w_2(\pi)=\dfrac{\chi_{\pi}(h_0)-\chi_{\pi}(h_1)}{4}a_2.$$

\subsubsection{For $\mathbf{q\equiv 3 \mod 4}$}

\begin{theorem}
	Let $\pi$ be a real representation of $\GL_n(\F_q)$, where $q\equiv 3\pmod 4$. Then 
	$$w_1(\pi)=\dfrac{\chi_{\pi}(h_0)-\chi_{\pi}(h_1)}{2}\left(\sum_{i=1}^n v_i\right).$$
\end{theorem}

\begin{proof}
	We have 
	$$w(\pi)=\prod_{i=1}^{n}\left(\prod_{b\in\X_i}(1+v\cdot b)\right)^{c_i}.$$
	We write $v\cdot b$ for $\sum\limits_{i=1}^nv_ib_i$. In particular,
	$$w_1(\pi)=\sum_{i=1}^{n}c_i\binom{n-1}{i-1}\left(\sum_{i=1}^{n} v_i\right)=Jm\left(\sum_{i=1}^{n} v_i\right).$$
	Now the proof follows from Proposition \ref{jm}.
\end{proof}

Now we calculate $w_2(\pi)$. For $b\in \X_i$ let $E_b=(1+v\cdot b)^{c_i}$ and $P_i=\prod_{b\in \X_i}E_b$. Note that 
$$(1+v\cdot b)^{c_i}=\left(1+c_i\left(\sum_{s=1}^{i}v_{j_s}\right)+\binom{c_i}{2}\left(\sum_{s=1}^{i}v_{j_s}^2\right)+2\binom{c_i}{2}\left(\sum_{k<l}v_{j_k}v_{j_l}\right)+T_3\right),$$
where $\{j_1,\ldots, j_i\}$ are such that $b_{j_s}=1$ and $T_3$ denotes the terms with degree greater than or equal to $3$. Let 
$N_i$ denote the number of elements in $\X_i$ with $1$ in a fixed position. We have $N_i=\binom{n-1}{i-1}$. Put $N_i'=\binom{n-2}{i-2}$ which is the number of elements in $\X_i$ with $1$ in two fixed positions. Also consider $N_i''=\binom{n-2}{i-1}$ to be the number of elements in $\X_i$ with $1$ in one fixed position and $0$ in other fixed position. 

The coefficient of $\sum_{t=1}^{n} v_{t}^2$ in the symmetric polynomial $P_i$ is
$$N_i\binom{c_i}{2}+c_i^2\binom{N_i}{2}.$$

The coefficient of $\sum_{1\leq l<k\leq n} v_kv_l$ in $P_i$ is 
\begin{equation}
	2\binom{n-2}{i-2}\binom{c_i}{2}+c_i^2N_i'(N_i'-1) + c_i^2 \cdot 2  N_i' N_i'' + c_i^2{N_i''}^2
\end{equation}
We ignore the term $2\binom{n-2}{i-2}\binom{c_i}{2}$ since it is even.

We calculate
\begin{align*}
	c_i^2N_i'(N_i'-1) + c_i^2 \cdot 2  N_i' N_i'' + c_i^2{N_i''}^2&=c_i^2\left({N_i'}^2 + 2N_i' N_i''+ {N_i''}^2 - N_i'\right)\\
	&=c_i^2\left((N_i' + N_i'')^2-N_i'\right)\\
	&=c_i^2\left(N_i^2-N_i'\right),\,\, (\text{using identity} N_i' + N_i'' = N_i)\\
	& = c_i\left(N_i-N_i'\right) \pmod 2\\
	&= c_i N_i'' \pmod 2.\\
\end{align*}

Since $w(\pi) = \prod_i P_i$, we have 
\begin{equation}
	\begin{split}
		w_2(\pi)&=\left(\sum_{1\leq i<j\leq n}N_iN_jc_ic_j+\sum_{i=1}^{n}N_i\binom{c_i}{2}+\sum_{i=1}^{n}c_i^2\binom{N_i}{2}\right)\left(\sum_{t=1}^{n} v_t^2\right)\\
		&+\left(\sum_{i=1}^{n}c_i N_i''+2\sum\limits_{1 \leq i < j \leq n}N_iN_jc_ic_j\right)\cdot\left(\sum_{k<l}v_kv_l\right). 
	\end{split}
\end{equation}

The Vandermonde identity says 
$$\binom{r_1+\cdots+r_t}{s}=\sum_{k_1+\cdots+k_t=s}\binom{r_1}{k_1}\cdots\binom{r_t}{k_t}.$$
Using this identity we calculate
\begin{align*}
	\binom{\sum_{i=1}^{n}c_iN_i}{2}&=\sum_{i=1}^{n}\binom{c_iN_i}{2}+\sum_{1\leq l<k\leq n}c_lc_kN_lN_k\\
	&=\sum_{i=1}^{n}\left(\sum_{k_1+\cdots+k_{N_i}=2}\binom{c_i}{k_1}\cdots\binom{c_i}{k_{N_i}}\right)+\sum_{1\leq l<k\leq n}c_lc_kN_lN_k\\
	&=\sum_{i=1}^{n}N_i\binom{c_i}{2}+\sum_{i=1}^{n}c_i^2\binom{N_i}{2}+\sum_{1\leq l<k\leq n}c_lc_kN_lN_k.
\end{align*}

Thus we have 
\begin{equation}\label{w2exp}
	w_2(\pi) = \dbinom{\sum\limits_{i=1}^n c_iN_i}{2} \left(\sum_{l=1}^n v_l^2\right) + \left(\sum_{i=1}^n c_iN_i'' \right) \left(\sum\limits_{1 \leq l < k \leq n}v_lv_k\right). 
\end{equation}

Put $N= [N_0'' , N_1'', N_2'', \ldots , N_n''].
$
So we have $Nm= \sum\limits_{j=1}^n c_iN_i'' $. We calculate $NM$.

\begin{prop} \label{Nm} 
	We have $\sum_i c_iN_i''= Nm =\dfrac{\chi_{\pi}(h_0)-\chi_{\pi}(h_2)}{4}$.
\end{prop}

\begin{proof} 
	
	The proof is similar to Proposition \ref{jm}.
\end{proof}

\begin{theorem}
	Let $\pi$ be a real representation of $\GL_n(\F_q)$, where $q \equiv 3 \pmod 4$. Then $$w_2(\pi) = \dbinom{m_\pi}{2} \sum_i v_i^2. $$
\end{theorem}
\begin{proof}
	Observe that $\sum_{i=1}^n c_iN_i = \dfrac{\chi_{\pi}(h_0) - \chi_{\pi}(h_1)}{2} = m_\pi$ by equation \eqref{eqjm} and Proposition \ref{jm}.
	From Equation \eqref{w2exp} and Proposition \ref{Nm} we have
	
	$$w_2(\pi) = \dbinom{m_\pi}{2} \sum_l v_l^2 + \dfrac{\chi_{\pi}(h_0) - \chi_{\pi}(h_2)}{4} \cdot \sum_{1 \leq l < k \leq n} v_l v_k .$$ 
	Now from \cite{ganjo}[Lemma 4.4]{} it follows that the coefficient of $\sum_{1 \leq l < k \leq n} v_lv_k$ is even. Hence we get the result.
\end{proof}

\subsection{Calculations of SWC Using Detection Results}

Let $i_k : \GL_k(\F_q) \hookrightarrow \GL_n(\F_q)$, where $k<n$,
denote the injection of $\GL_k(\F_q)$ inside $\GL_n(\F_q)$ at the top left corner with rest diagonal entries $1$. When $q \equiv 1 \pmod 4 $ we have $i_1^*(a_2) = t_1$ and when $q \equiv 3 \pmod 4$ we have $i_2^*(\sum_{i=1}^n v_i^2)= v_1^2 + v_2^2$.

%Note that $i_\alpha^*(b_1)=s_1+s_2$ and $j^*(\sum\limits_{i=1}^nv_i)=v_1+v_2$.

\begin{comment} 
then  we have
$$ \pi\mid_{C_2} = r_0 \mathbb{1} \oplus r_1 \sgn$$
or
$$\pi\mid_{C_2\times C_2}=r_0\mathbb1\oplus r_1(\sgn\otimes \mathbb1\oplus \mathbb1\otimes \sgn)\oplus r_2(\sgn\otimes \sgn).$$
\end{comment}
For a real representation $\pi$ of $\GL_1(\F_q)$, Theorem \ref{main} gives
\[
w(\pi)=\begin{cases} 
	(1+\delta s_1)(1+t_1)^{c_1/2}, &\text{ if } q\equiv 1 \pmod 4,\\
	(1+v_1)^{c_1}, & \text{ if } q \equiv 3 \pmod 4,
\end{cases}
\]
For a real representation $\pi$ of $\GL_2(\F_q)$, Theorem \ref{main} says
\[
w(\pi)=\begin{cases} 
	(1+\delta(s_1+s_2))((1+t_1)(1+t_2))^{c_1'/2}(1+t_1+t_2)^{c_2'/2}, &\text{ if } q\equiv 1 \pmod 4,\\
	(1+v_1)^{c_1'}(1+v_2)^{c_1'}(1+v_1+v_2)^{c_2'}, & \text{ if } q \equiv 3 \pmod 4.
\end{cases}
\]

\vskip 1mm

%\begin{lemma}
%	For a real representation $\pi$ of $\GL_1(\F_q)$ we have
%	\begin{itemize}
	%		\item $r_1 = (1/2)(\chi_{\pi}(h_0) - \chi_{\pi}(h_1))= m_\pi$
	%	\end{itemize}
%	and for that of $\GL_2(\F_q)$ we have
%	\begin{itemize}
	%		\item $r_1= (1/4)(\chi_{\pi}(h_0) - \chi_{\pi}(h_2))$
	%		\item $r_2 = (1/4)(\chi_{\pi}(h_0) - 2\chi_{\pi}(h_1)+ \chi_{\pi}(h_2))$
	%	\end{itemize}
%\end{lemma}
For the case of $\GL_1(\F_q)$	
we have $$ \begin{pmatrix}
	c_0\\
	c_1\\
\end{pmatrix}= (1/2)\begin{pmatrix}
	1 & 1\\
	1 & -1 \\
	
\end{pmatrix}\cdot \begin{pmatrix}
	\chi_{\pi}(h_0)\\
	\chi_{\pi}(h_1)\\
\end{pmatrix},$$
while for the case of $\GL_2(\F_q)$ we have	
$$ \begin{pmatrix}
	c_0'\\
	c_1'\\
	c_2'\\
\end{pmatrix}= (1/4)\begin{pmatrix}
	1 & 2 & 1\\
	1 & 0 & -1 \\
	1 & -2 & 1\\
\end{pmatrix}\cdot \begin{pmatrix}
	\chi_{\pi}(h_0)\\
	\chi_{\pi}(h_1)\\
	\chi_{\pi}(h_2)\\
\end{pmatrix}.$$

For $q\equiv 3\pmod 4$ we compute $w_{1}(\pi)$ in terms of character values.

\begin{theorem}\label{w1ofpi} 
	Let $\pi$ be a real representation of $\GL_n(\F_q)$, where $q\equiv 3\pmod 4$. Then we have
	$$w_1(\pi)=m_{\pi}\left(\sum_{i=1}^n v_i\right).$$
	%$$w_1(\pi)=\dfrac{\chi_{\pi}(h_0)-\chi_{\pi}(h_1)}{2}\left(\sum_{i=1}^n v_i\right).$$
\end{theorem}

\begin{proof}
	
	Since $i=1$ in Theorem \ref{detectq3}, we have $n = 2$. We calculate for $\GL_1(\F_q)$. So for the representation $\pi$ we have 
	
	$$w_1(\pi) = c_1v_1 .$$
	Since $c_1= m_\pi$, we have the result.
	
	%For $q\equiv 1\pmod 4$ we have
	%$$w_1(\sigma_i)=\sum_{1\leq j_1<j_2<\ldots<j_i\leq n}(s_{j_1}+\cdots+s_{j_i})=\binom{n-1}{i-1}b_1.$$
	%Then we calculate
	%\begin{align*}
	%w_1(\pi)&=w_1\left(\bigoplus\limits_{i=1}^{n}m_i\sigma_i\right)\\
	%&=\sum_{i=1}^{n}m_iw_1(\sigma_i)\\
	%&=\sum_{i=1}^{n}m_i\binom{n-1}{i-1}b_1.
	%\end{align*}
	%Hence $\delta= \sum\limits_{i=1}^{n}m_i\binom{n-1}{i-1}.$ From Lemma \ref{jm} we conclude that $\delta= m_\pi$.	
	%
	%Now consider the case $q\equiv 3\pmod 4$. We write $v\cdot x$ for $\sum\limits_{i=1}^nv_ix_i$. In particular,
	%$$w_1(\pi)=\sum_{i=1}^{n}m_i\binom{n-1}{i-1}\left(\sum_{i=1}^{n} v_i\right)=Jm\left(\sum_{i=1}^{n} v_i\right).$$
	%Now the proof follows from Lemma \ref{jm}.
	
\end{proof}

Next we compute the second SWC.

\begin{theorem}\label{w2ofpi}
	Let $\pi$ be a real representation of $\GL_n(\F_q)$, where $q$ is odd. Then we have
	\[
	w_2(\pi)=
	\begin{cases}
		\dfrac{m_{\pi}}{2}a_2,&\text{if}\,\, q\equiv 1\pmod 4\\ \\
		\dbinom{m_{\pi}}{2}\left(\sum\limits_{i=1}^n v_i^2\right), &\text{if}\,\, q\equiv 3\pmod 4.
	\end{cases}
	\]
	%Let $\pi$ be a real representation of $\GL_n(\F_q)$ where $q\equiv 1\pmod 4$. Then $\pi$ is spinorial iff
	%$$\sum\limits_{j=1}^k (r_j/2)(\binom{k-1}{j-1})a_2\equiv 0\pmod 2,$$
	%where $a_2=\sum\limits_{i=1}^k t_i$.
\end{theorem}

\begin{proof}
	For a real representation $\pi$ of $\GL_1(\F_q)$, where $q\equiv 1\pmod 4$, we have
	
	$$ w_2(\pi) = \frac{c_1}{2} t_1. $$

	We have $c_1 = (1/2)(\chi_{\pi}(h_0) - \chi_{\pi}(h_1)) = m_\pi.$ 
	Thus by Theorem \ref{detectq1} we have $$w_2(\pi) = \frac{m_\pi}{2}a_2.$$

	%For a real representation $\pi$ of $\GL_2(\F_q)$ where $q\equiv 3\pmod 4$ we compute
	%$$
	%w_2(\pi)= \frac{(r_1+r_2)}{2}(v_1^2+v_2^2).
	%$$
	For a real representation $\pi$ of $\GL_2(\F_q)$ and $q \equiv 3 \pmod{4}$ we have 
	$$w_2(\pi) = \binom{c_1' + c_2'}{2}(v_1^2 + v_2^2 ) + c_1'^2 v_1v_2.$$
	We know from \cite[Lemma $4.4$]{ganjo}) that $c_1'$ is even. 
Moreover we have 
	$	c_1' +c_2' = (1/2)(\chi_{\pi}(h_0)-\chi_{\pi}(h_1))=m_\pi$.
	%Thus $w_2(\pi)=(1/2)(\chi_{\pi}(h_0)-\chi_{\pi}(h_1)) $.
Hence the Theorem follows from Theorem \ref{detectq3}.	
\end{proof}

Using the detection result we compute the fourth SWC in terms of character values when $q\equiv 1\pmod 4$.

\begin{theorem}\label{w4}
	Let $\pi$ be a real representation of $\GL_n(\F_q)$, where $q\equiv 1\pmod 4$. Then 
	$$w_4(\pi)=\binom{m_{\pi}/2}{2}\sum_{i=1}^nt_i^2+\frac{\dim \pi-\chi_{\pi}(h_2)}{8}\sum_{1\leq i<j\leq n}t_it_j.$$
\end{theorem}

\begin{proof}
	From Theorem \ref{detectq1}
	we know that the map $\res:H^4(\GL_n(\F_q))\to H^4(\GL_2(\F_q))$ is injective.  We have
	$$w^{\GL_2(\F_q)}(\pi)=(1+\delta(s_1+s_2))\left((1+t_1)(1+t_2)\right)^{c_1'/2}(1+t_1+t_2)^{c_2'/2}.$$
	Then we calculate
	\begin{align*}
		\res(w_4(\pi))&=\left(\binom{c_1'/2}{2}+\binom{c_2'/2}{2}+\frac{c_1'c_2'}{4}\right)(t_1^2+t_2^2)+\frac{c_1'}{2}t_1t_2\\
		&=\binom{\frac{c_1'}{2}+\frac{c_2'}{2}}{2}(t_1^2+t_2^2)+\frac{c_1'}{2}t_1t_2\\
		&=\binom{m_{\pi}/2}{2}(t_1^2+t_2^2)+\frac{\dim \pi-\chi_{\pi}(h_2)}{8}t_1t_2.
	\end{align*}
	
	So we have 
	$$w_4(\pi)=\binom{m_{\pi}/2}{2}\sum_{i=1}^nt_i^2+\frac{\dim \pi-\chi_{\pi}(h_2)}{8}\sum_{1\leq i<j\leq n}t_it_j.
	$$
\end{proof}

\section{Some Applications}\label{appl}

\subsection{Lifting Criteria for Real Representations of $\GL_n(\F_q)$}

A real representation $\pi$ of a finite group $G$ is called spinorial provided it lifts to the $\mathrm{Pin}$ group, which is a double cover of the orthogonal group. One knows that $\pi$ is spinorial iff $w_2(\pi)+w_1^2(\pi)=0$ (see \cite[page $328$]{guna}).

		\begin{theorem}\label{lift}
			Let $\pi$ be a real representation of $\GL_n(\F_q)$, where $q$ is odd. Then  
			\begin{enumerate}
				\item 
				for $q\equiv 1\pmod 4$, we have $\pi$ is spinorial iff $m_{\pi}\equiv 0\pmod 4$,
				\item
				for $q\equiv 3\pmod 4$, we have $\pi$ is spinorial iff $m_{\pi}\equiv 0$ or $3\pmod 4$.
			\end{enumerate}
		\end{theorem}
		
		\begin{proof}
			We know that $\pi$ is spinorial iff $w_2(\pi)+w_1^2(\pi)=0$. From Theorems \ref{detectq1} and \ref{detectq3} we deduce that $\pi$ is spinorial iff $res(w_2(\pi)+w_1^2(\pi))=0$, where $\res:H^*(\GL_n(\F_q),\ztwo)\to H^*(\GL_k(\F_q),\ztwo)$ denotes the restriction map, where $k=1$ for $q \equiv 1 \pmod 4$ and $k=2$ for $q \equiv 3 \pmod 4$. 
			
			%For a real representation $\pi$ of $\GL_2(\F_q)$ where $q\equiv 1\pmod 4$ we have $w_1^2(\pi)=0$ as $s_i^2=0$ for all $i$. Therefore if follows that $\pi$ is spinorial if $w_2(\pi)=0$.
			
			%For $q \equiv 3 \pmod 4$ we have $w_1(\pi)=(r_1+r_2)(v_1+v_2)$ and hence by Theorem \ref{w1ofpi} we obtain $r_1+r_2 = m_\pi$.
			%We have  
			%$$w_2(\pi)=\left(\binom{m_\pi}{2}\right)(v_1^2+v_2^2).$$
			%We obtain
			%$$w_2(\pi)+w_1^2(\pi)=\left(\binom{m_\pi}{2} + (m_\pi)^2\right)(v_1^2+v_2^2).$$
			%We know that $r_1$ is even when $q\equiv 3\pmod 4$ (see \cite[Lemma $5$]{ganjo}). 
			%Thus we have
			%$$w_2(\pi)+w_1^2(\pi)=\left((m_\pi)(m_\pi+1)/2\right)(v_1^2+v_2^2).$$
			%Finally observe that $r_1+r_2 = m_\pi$.
			
			From Theorems \ref{w1ofpi} and \ref{w2ofpi}, it follows that for a real representation $\pi$ of $\GL_n(\F_q)$, we have 
			\[ 
			\res(w_2(\pi) + w_1^2(\pi))=
			\begin{cases}
				\dfrac{m_\pi}{2}t_1,  &\text{ if } q \equiv 1 \pmod 4,\\
				\dfrac{m_{\pi}(m_\pi + 1)}{2} (v_1^2+v_2^2), & \text{ if } q \equiv 3 \pmod 4.
			\end{cases}
			\]

		\end{proof}

\subsection{Non-Detection by An-isotropic Torus}
Let $a \in \F_{q^n}^{\times}$ and $G=\GL_n(\F_q)$, where $q$ is odd. We fix a basis of $\F_{q^n}$ over $\F_q$. Then the map $x \to a\cdot x$ gives a linear map $M_a \in G  $. Then $M=\{M_a, a \in \F_{q^n}^{\times} \}$ is a cyclic group of order $q^n-1$ and is called an anisotropic torus. 
\begin{theorem}\label{app1}
	The anisotropic torus $M$ does not detect the mod-2 cohomology of $G$. 	
\end{theorem} 
\begin{proof}
	Consider the representation $\pi = (\chi)_{\R}$, where $\chi=\chi^j\circ\det$ and $j$ is odd. 
One has 
\begin{align*}
\chi_{\pi}(h_1)&=2\chi_j(\det(h_1))\\
&=2\chi_j(-1)\\
&=2(-1)^j.
\end{align*}
Since $j$ is odd we have $\chi_{\pi}(h_1)=-2$. This gives 
$$m_{\pi}=\frac{1}{2}(2+2)=2.$$
	Then from Theorem \ref{w2ofpi} one obtains 
	$w_2(\pi)\neq 0.$
	
	Let $\hat g$ be a generator of $M$. Now we have
	
	$$w_2(\pi\mid_M)= \kappa(c_1(\chi^j \circ \det\mid_M)).$$
	We have
	$$\chi(\hat{g})=\chi^j \circ \det (\hat{g})=\zeta_{q-1}^j=\zeta_{q^2-1}^{(q+1)j}.$$ 
	Take $u=c_1(\dot{\chi})\in H^2(M,\Z)$ where $\dot{\chi} \in \mathrm{Hom}(M, \C^{\times})$ such that $\dot{\chi}(\hat{g}) = \zeta_{q^2-1}$. As a result we get 
	\begin{align*}
		\kappa(c_1(\chi^j \circ \det\mid_M))&= \kappa((q+1)ju)\\
		&=(q+1)j \kappa(u)\\
		&=0.
	\end{align*}
	This shows that $w_2(\pi\mid_M)=\res^G_M (w_2(\pi))=0$ whereas $w_2(\pi)\neq 0$. 	
\end{proof}

\section{The Case of $\GL_n(\C)$ and $\GL_n(\R)$}\label{cr}
For a Lie group $G$ we write $H^i(G)$ to denote the (singular) cohomology $H^i(BG)$ of a classifying space $BG$ of $G$. Consider the subgroup
$$\Gamma=\diag(\pm 1, \pm 1, \ldots, \pm 1)\cong C_2^n$$
of $\GL_n(k)$, where $k = \R \text{ or } \C$.
Take $T=(S^1)^n$ to be the diagonal torus of the unitary group $U_n$. Now we establish some useful detection results. 

\begin{prop}
The mod $2$ cohomology of $\GL_n(\C)$ is detected by the group $\Gamma$.
\end{prop}

\begin{proof}
%\begin{center}
%		\begin{tikzpicture}
%		\node (A1) at (0,2) {$H^*(\GL_n(\C)) $};
%		\node (A2) at (4,2) {$H^*(U_n)$};
%		\node (B1) at (0,0) {$H^*(T_{\C^{\times}})$};
%		\node (B2) at (4,0) {$H^*(T_{S^1})$};
		
		%\path[->,font=\scriptsize,>=angle 90]
%		\draw[->](A1)to node [above]{$i_{DG_n}^*$} (A2);
%		\draw[->](A1) to node [right]{$i_{GG}^*$} (B1);
%		\draw[->](B1) to node [above]{$i_{DG_2}^*$} (B2);
%		\draw[->](A2) to node [right]{$i_{DD}^*$} (B2);
%		\end{tikzpicture}
%	\end{center}
From \cite[Theorem $2.1$]{toda} we obtain that the maps $i_1^*:H^*(\GL_n(\C))\to H^*(U_n)$ and $i_2^*:H^*(U_n)\to H^*(T)$ are injective. As a result we conclude that the mod $2$ cohomology of $\GL_n(\C)$ is detected by the diagonal torus $T$.
First we prove that the mod $2$ cohomology of $S^1$ is detected by $C_2$. Note that the circle group $S^1$ is isomorphic to $\mathrm{SO}_2(\mb R)$. We know that $H^*(\mathrm{SO}_2(\mb R))=\ztwo[w_2]$, where $w_2$ denotes the second Stiefel Whitney class of standard representation of $\mathrm{SO}_2(\mb R)$. The standard representation of $\mathrm{SO}_2(\mb R)$ can be thought of as the representation $\pi=\rho\oplus\rho^{\vee}$ of $S^1$ where $\rho(z)=z$ for $z\in S^1$. So $w_2=w_2(\pi)$. One can consider $C_2=\{\pm1\}$ as a subgroup of $S^1$ to obtain a map $\res:H^*(\mathrm{SO}_2(\R))\to H^*(C_2,\ztwo)$. We have $H^*(C_2,\ztwo)=\ztwo[v]$, where $\deg v=1$.  Note that $\pi\mid_{C_2}=\sgn\oplus \sgn$. Thus we have 
$$\res w_2(\pi)=w_2(\sgn\oplus\sgn)=w_1^2(\sgn)\neq0.$$
So $\res(w_2)=v^2$. This implies that the map `$\res$' is injective. Now the result follows from Kunneth formula.

\end{proof}

%We have 
%\begin{equation}\label{res}
%\pi\mid_{\Gamma}=m_0\mathbb1\bigoplus_{i=1}^{n} m_i\sigma_i,
%\end{equation}
%where $\sigma_i=\bigoplus\limits_{b\in\X_i}\sgn_b$.

%We leverage this result to compute the total Stiefel Whitney class of a real representation of $\GL_2(\C)$.
%
%
%
%\begin{proof}
%
%We compute $\res w(\pi)$. The restriction of $\pi$ to $\Gamma$ takes the form as in \eqref{res} and the theorem follows from Proposition \ref{wpiv2n}.
%\end{proof}

\begin{prop}\label{gamma}
	The mod $2$ cohomology of $GL_n(\R)$ is detected by $\Gamma$.
\end{prop}
\begin{proof}
From \cite[Theorem $7.7$]{mitchell} we know that $B\GL_n(\R)$ is homotopy equivalent to $BO_n(\R)$. Thus the restriction map  $\res_{O}:H^*(\GL_n(\R)) \to H^*(O_n(\R))$ is an isomorphism. From \cite[Theorem 2.2]{toda} we obtain that $\res_{\Gamma}: H^*(O_n(\R)) \to H^*(\Gamma)$ is an injection. Thus the $\ztwo$ cohomology of $\GL_n(\R)$ is detected by the subgroup $\Gamma \cong C_2^n$.  
\end{proof}

\begin{proof}[Proof of Theorem \ref{RandC}]
Given a real representation $\pi$ of $\GL_n(\R) (\text{ or }\GL_n(\C))$, we calculate the total Stiefel Whitney class of $\psi =\pi\mid_\Gamma$. Moreover $\psi$ is a $S_n$ invariant representation as it is a restriction of representation of $\GL_n(\R) (\text{ or }\GL_n(\C))$. Now the Theorem \ref{RandC} follows from Theorem \ref{wpiv2n}.
\end{proof}

\begin{remark}
Let $\pi$ be a representation of $\mathrm{O}_n(\R)$. Then the proof of Proposition \ref{gamma} gives that the total SWC of $\pi$ is same as the expression in Equation \eqref{on}. 
\end{remark}

\section{Examples}\label{exam}

We have 

\[
w_2(\pi)=
\begin{cases}
	\dfrac{m_{\pi}}{2}a_2, & \text{ if } q\equiv 1 \mod 4, \\ \\
	\dbinom{m_\pi}{2}(\sum v_i^2),& \text{ if } q\equiv 3 \mod 4 ,
\end{cases}
\]	
where $m_\pi = (1/2)(\dim \pi - \chi_\pi(h_1))$. Also 
  $$w_4(\pi)=\binom{m_\pi/2}{2} \sum_i t_i^2 + \frac{n_\pi}{4} \sum_{i<j} t_it_j,$$
where $n_\pi = \dfrac{\dim \pi - \chi_{\pi}(h_2)}{2}$.

\subsection{Principal Series Representation}\label{prin}

Let $\chi_1, \chi_2,\ldots, \chi_n$ are linear characters of $\F_q^{\times}$ and let $B$ be the Borel subgroup of $\GL_n(\F_q)$ . Put 
$$\pi=\mathrm{Ind}_B^G(\chi_1\boxtimes \chi_2\boxtimes\cdots\boxtimes \chi_n).$$
Such representations are called principal series representations of $G$. Such a representation $\pi$ is orthogonal iff
$$\{\chi_1,\ldots,\chi_n\}=\{\chi_1^{-1},\ldots,\chi_n^{-1}\},$$
as multisets. We introduce the notation 

 \begin{align*}
 [n]_q! &= \prod_{i=1}^{n}\dfrac{q^i-1}{q-1}\\
 &=(q+1)(q^2+q+1)\cdots(q^{n-1}+q^{n-2}+\cdots+q+1).
 \end{align*}

From \cite[Proposition $7.5.3$, page no. $233$]{carter} we obtain  the following result.

	\begin{equation}
	\chi_{\pi}(h_k) = [k]_q![n-k]_q! \sum_{1 \leq i_1 < i_2 < \cdots <i_k \leq n}\chi_{i_1}(-1)\chi_{i_2}(-1) \cdots \chi_{i_k}(-1).
	\end{equation}

Put
$$T_i=(1+q+q^2+\cdots+q^i).$$

So we calculate
\begin{align*}
m_{\pi}&=\frac{1}{2}(\dim\pi-\chi_{\pi}(h_1))\\
&=\frac{1}{2}\left([n]_q!-[n-1]_q!\left(\sum_{i=1}^n\chi_i(-1)\right) \right)\\
&=\frac{1}{2}[n-1]_q!\left(T_{n-1}- \sum_{i=1}^n\chi_i(-1)\right).
\end{align*}

Similarly one has
\begin{align*}
n_{\pi}&=\frac{1}{2}(\dim\pi-\chi_{\pi}(h_2))\\
&=\frac{1}{2}[n-2]_q!\left(T_{n-2}T_{n-1}-(1+q)\sum_{i<j}\chi_i(-1)\chi_j(-1) \right).
\end{align*}

\begin{theorem}
Let $\pi$ be a real principal series representation for $\GL_n(\F_q)$ where $q$ is odd. Then $w_2(\pi)=0$ for $n\geq 3$.
\end{theorem}

\begin{proof}
From \cite[Thorem $5$]{spjoshi} we know that if $n\geq 3$, then all the real principal series representations of $\GL_n(\F_q)$ are spinorial. Now $\pi$ is spinorial iff $w_2(\pi)=w_1^2(\pi)$ (see \cite[Proposition $6.1$]{jyoti}). If $q\equiv 1\pmod 4$,then $w_1^2(\pi)=0$. Therefore $w_2(\pi)=0$.

Now consider the case for $q\equiv 3\pmod 4$.
We have
$$m_{\pi}=\frac{1}{2}[n-1]_q!\left(T_{n-1}-\sum_{i=1}^n\chi_{i}(-1)\right).$$
Also
$$T_{n-1}\equiv \sum_{i=1}^n\chi_{i}(-1)\equiv n\pmod 2.$$
This gives $T_{n-1}-\sum\limits_{i=1}^n\chi_{i}(-1)\equiv 0\pmod 2$.
We have $[m]_q! \equiv 0 \pmod 4 $ for $m\geq 2$. For $n\geq 3$, we have $\upsilon_2(m_{\pi})\geq 2$. Therefore
$\dbinom{m_{\pi}}{2}\equiv 0\pmod 2$.
Therefore $w_2(\pi)=0$ for $n\geq 3$.

\end{proof}

\begin{proof}[Proof of Theorem \ref{psr}]

We have
\begin{align*}
n_{\pi}&=\frac{1}{2}[n-2]_q!\left(T_{n-2}T_{n-1}-(1+q)\sum_{i<j}\chi_i(-1)\chi_j(-1) \right).
\end{align*}
Observe that $T_{n-2}T_{n-1}-(1+q)\sum_{i<j}\chi_i(-1)\chi_j(-1)$ is always even. Also $\upsilon_2([n-2]_q!)\geq 3$ for $n\geq 6$. Therefore $n_{\pi}/4$ remains even for $n\geq 6$.

We have
$$m_{\pi}=\frac{1}{2}[n-1]_q!\left(T_{n-1}-\sum_{i=1}^n\chi_{i}(-1)\right).$$
Also 
$$T_{n-1}\equiv \sum_{i=1}^n\chi_{i}(-1)\equiv n\pmod 2.$$
This gives $T_{n-1}-\sum\limits_{i=1}^n\chi_{i}(-1)\equiv 0\pmod 2$. For $n\geq 6$, the term $m_{\pi}/2$ remains even. So we have $\dbinom{m_{\pi}/2}{2}\equiv \dfrac{m_{\pi}}{4}\pmod 2$. Now $\upsilon_2(m_{\pi})\geq 4$ for $n\geq 6$. Therefore $w_4(\pi)=0$ for $n\geq 6$.

For $n=5$ the two possible real Principal series representations are $\pi_1=\{\mathbb{1}, \chi_i,\chi_i^{-1},\chi_j,\chi_j^{-1}\}$ and $\pi_2=\{\sgn, \chi_i,\chi_i^{-1},\chi_j,\chi_j^{-1}\}$.

We calculate
\begin{align*}
m_{\pi_1}&=\frac{1}{2}\left([5]_q!-[4]_q!(1+2(-1)^i+2(-1)^j) \right)\\
&=\frac{1}{2}(q+1)^2(q^2+1)(q^2+q+1)\left(q(q+1)(q^2+1)-2(-1)^i-2(-1)^j \right).
\end{align*}
This shows $\upsilon_2(m_{\pi_1})\geq 4$. Similarly one obtains $\upsilon_2(m_{\pi_2})\geq 4$.

We have
\begin{align*}
n_{\pi_1}&=\frac{1}{2}\left([5]_q!-[2]_q![3]_q!(2+2(-1)^i+2(-1)^j+4(-1)^{i+j})  \right)\\
&=\frac{1}{2}(q+1)^2(q^2+q+1)\left((q^2+1)(q^4+q^3+q^2+q+1)-(2+2(-1)^i+2(-1)^j+4(-1)^{i+j}) \right)\\
&=\frac{1}{2}(q+1)^2(q^2+q+1)\left(\{q^6+q^5+2(q^4+q^3+q^2)\}+\{q-1\}-\{2(-1)^i+2(-1)^j)+4(-1)^{i+j}\} \right).
\end{align*}
Note that the all the terms in the curly brackets and $(q+1)^2$ are divisible by $4$. Therefore $\upsilon_2(n_{\pi_1})\geq 3$.
Similar calculation shows that $\upsilon_2(n_{\pi_2})\geq 3$. Now the result follows from Theorem \ref{w4}.
\end{proof}

Here we list down the value of $w_4(\pi)$, where $\pi$ is an orthogonal principal series representation of $\GL_n(\F_q)$ for $q\equiv 1\pmod 4$.

\begin{enumerate}
\item
Take $n=3$ and $\pi=\{\mathbb{1},\chi_j,\chi_j^{-1}\}$ or $\{\sgn,\chi_j,\chi_j^{-1}\}$ Then
$$w_4(\pi)=\sum_{i=1}^3 t_i^2+\sum_{1\leq i<j\leq 3} t_it_j$$ 
if $q\equiv 1,9\pmod {16}$ and $j$ is odd or $q\equiv 5,13\pmod{16}$ and $j$ is even. Otherwise $w_4(\pi)=0$.
\item
Take $n=4$ and $\pi=\{\mathbb{1},\sgn,\chi_i,\chi_i^{-1}\}$. Then
\[
w_4(\pi)=
\begin{cases}
\sum\limits_{i=1}^4 t_i^2,& \text{if $i$ is odd},\\
0, & \text{if $i$ is even}.
\end{cases}
\]
\item
Take $n=4$ and $\pi=\{\chi_j,\chi_j^{-1},\chi_k,\chi_k^{-1}\}$.
\[
w_4(\pi)=
\begin{cases}
\sum\limits_{i=1}^4 t_i^2,&\text{if $j$ and $k$ have different parity},\\
0, & \text{otherwise}.
\end{cases}
\]

\end{enumerate}

%\begin{center}
%	\begin{table}[ht]
%	\caption{ Computation of $w_4(\pi)$ where $\pi=\{\mathbb{1},\chi_i,\chi_i^{-1}\}$ or $\{\sgn, \chi_i, \chi_i^{-1}\}$  is a principal series representation of $\GL_n(\F_q)$ for $n=3$ and $q\equiv 1\pmod4$}
%	\centering	
%	
%	\begin{tabular}{|c|c|c|c|}
%	\hline
%	$q\pmod{16}$ & $i\pmod2$ & $\binom{m_{\pi}/2}{2}\pmod 2$ & $n_{\pi}/4\pmod 2$\\
%	\hline
%	$1$ & $0$ & $0$& $0$\\
%	\hline
%	$1$ & $1$ &$1$ & $1$\\
%	\hline
%	$5$ & $0$ & $1$ & $1$\\
%	\hline
%	$5$ & $1$ & $0$& $0$ \\
%	\hline
%	$9$ & $0$ & $0$ & $0$\\
%	\hline
%	$9$ & $1$ & $1$ & $1$\\
%	\hline
%	$13$ & $0$ &$1$ &$1$ \\
%	\hline
%	$13$ & $1$ & $0$& $0$ \\
%	\hline
%	
%	\end{tabular}
%	\label{table1}
%\end{table}
%\end{center}

\subsection{Cuspidal Representations}

Let $\theta$ be a regular character of full anisotropic torus $T$ of $\GL_n(\F_q)$. Then there is a cuspidal representation $\pi_{\theta}$ associated to it. The representation $\pi_{\theta}$ is real iff there exists a Weyl group element $w$ such that $w(\theta)=\theta^{-1}$.
 
Let $\pi$ be a real Cuspidal representation of $\GL_n(\F_q)$.
\subsubsection{\bf{Computation of $w_2(\pi)$}}

We know that $\chi_\pi(h_1) = 0$ and $\dim \pi = \psi_{n-1}(q)$ where $$ \psi_m(q)=\prod_{i=1}^{m}(q^i-1). $$ Thus $m_\pi = \psi_{n-1}(q)/2$.

\begin{theorem}
Let $\pi$ be a real Cuspidal representation of $\GL_n(\F_q)$ such that $n\geq 3$. Then $w_2(\pi)=0$.	
\end{theorem}

\begin{proof}
For $n\geq 3$, the term $\dfrac{(q-1)(q^2-1)}{2}$ divides $m_\pi$, Thus $m_\pi \equiv 0 \mod 4$, and therefore $w_2(\pi) = 0$.
\end{proof}

For $n =2$ and $q \equiv 1 \mod{4}$, we have $m_\pi = (q-1)/2$ and $w_2= \frac{q-1}{4} a_2$.
\[ w_2(\pi) = 
\begin{cases}
	0 ,  &\text{ if } q \equiv 1 \mod 8,\\
	a_2,   &\text{ if } q \equiv 5 \mod 8. \\
\end{cases}
\]

For $n =2$ and $q \equiv 3 \mod 4$, we have $m_\pi = (q-1)/2$ and $w_2=\binom{m_\pi}{2}\cdot  ( v_1^2 + v_2^2)$.
Therefore 
\[ w_2(\pi) = 
\begin{cases}
	0 ,  &\text{ if } q \equiv 3 \mod 8,\\
	v_1^2 + v_2^2 ,  &\text{ if } q \equiv 7 \mod 8. \\
\end{cases}
\]
\subsubsection{\bf{Computation of $w_4(\pi)$}}

We summerize the computations in the following table.
%\begin{center}
%	\begin{table}[ht]
%		\caption{ Computation of $w_4(\pi)$ where $\pi=\{\mathbb{1},\chi_i,\chi_i^{-1}\}$ or $\{\sgn, \chi_i, \chi_i^{-1}\}$  is a principal series representation of $\GL_n(\F_q)$ for $n=3$ and $q\equiv 1\pmod4$}
%		\centering	
%		
%		\begin{tabular}{|c|c|c|c|}
%			\hline
%			$q \mod 16$ & $\theta(-1)$ & $\binom{m_\pi/2}{2} \pmod 2$&$n_\pi/4 \pmod 2$\\
%			\hline 
%			$1$& 1 & 0 &0\\
%			\hline  
%			$1$ &
%			$-1$ & 0 &0\\  
%			\hline 
%			$5 $ & $ 1$ & 0 &0\\
%			\hline
%			$ 5$& $ -1$ & 0 &1\\
%			\hline 
%			$ 9$& $ 1$ &1&0\\
%			\hline 
%			$ 9$ & $-1$ & 1&0\\
%			\hline 
%			$ 13$ & $1$ & 1 &0\\
%			\hline 
%			$ 13$ & $ -1$ &1&1\\
%			\hline 
%%		\end{tabular}
%%	\label{cusp.table2}
%\end{tabular}
%\label{cusp.table}
%\end{table}
%\end{center}

%We have  $$w_4(\pi)=\binom{m_\pi/2}{2} \sum_i t_i^2 + \frac{n_\pi}{4} \sum_{i<j} t_it_j,$$
%where $n_\pi = \dfrac{\dim \pi - \chi_{\pi}(h_2)}{2}$.

	\begin{center}
	\begin{table}[ht]
		\caption{Computation of $w_4(\pi)$ where $\pi$ is the Cuspidal representation of $\GL_n(\mb F_q)$}
		\centering	
		
		\tabcolsep=.6cm
		\renewcommand{\arraystretch}{2}
		\begin{tabular}{|p{1.5cm}|p{1.5cm}|p{1.5cm}|p{1.5cm}|}
			\hline  
			$q \mod 16$ & $\theta(-1)$ & $\binom{m_\pi/2}{2} \pmod 2$&$n_\pi/4 \pmod 2$\\
			\hline 
			$1$& 1 & 0 &0\\
			\hline  
			$1$ &
			$-1$ & 0 &0\\  
			\hline
			$5 $ & $ 1$ & 0 &0\\
			\hline
			$ 5$& $ -1$ & 0 &1\\
			\hline 
			$ 9$& $ 1$ &1&0\\
			\hline 
			$ 9$ & $-1$ & 1&0\\
			\hline 
			$ 13$ & $1$ & 1 &0\\
			\hline 
			$ 13$ & $ -1$ &1&1\\
			\hline 
		\end{tabular}

		\label{cusp.table2}
	\end{table}
\end{center}

\subsection{Steinberg  Representation}
Let $\pi$ be the Steinberg representation of $\GL_n(\F_q)$.

\subsubsection{\bf{Computation of $w_2(\pi)$}}

%Here $m_\pi = (1/2)(\dim \pi - \chi_\pi(h_1))$.
%
%
%We have 
%
%\[
%w_2(\pi)=
%\begin{cases}
%	(m_\pi /2)\cdot  a_2 \text{ if } q\equiv 1 \mod 4 \\
%	\binom{m_\pi}{2}(\sum v_i^2) \text{ if } q\equiv 3 \mod 4 .
%\end{cases}
%\]

We have $\chi_\pi(h_1)=p^k$ where $p^k$ is the highest power of $p$ dividing $|Z_G(h_1)|$ (see \cite[Theorem $6.4.7$, page $195$]{carter}). We have $Z_G(h_1) \cong \GL_1(\F_q) \times \GL_{n-1}(\F_q) $. Thus $\chi_\pi(h_1) = q^{(1/2)(n-1)(n-2)}$.
Therefore 
\begin{align*}
	m_\pi &= \dfrac{q^{(1/2)(n(n-1))} - q^{(1/2)(n-1)(n-2)}}{2}\\
	&=(1/2)\cdot q^{(1/2)(n-1)(n-2)}(q^{n-1}-1)\\
	&=\dfrac{q-1}{2} \cdot q^{(1/2)(n-1)(n-2)} \cdot (1+q+q^2+ \cdots + q^{n-2}).
\end{align*}

For $ q \equiv 1 \pmod 4$ we have

%\[ m_\pi = 
%\begin{cases}
%	0 \mod{4}  &\text{ if } q \equiv 1 \mod 8\\
%	0 \mod{4}  &\text{ if } q \equiv 5 \mod 8 \text{ and for } n \text{ odd } \\
%	2 \mod{4} &\text{ if } q \equiv 5 \mod 8 \text{ and for } n \text{ even }
%\end{cases}
%\]

\[w_2(\pi)=
\begin{cases}
	a_2, &\text{ if } q \equiv 5 \mod 8 \text{ and for } n \text{ even },\\
	0, &\text{  otherwise }.
\end{cases}
\]

For $q \equiv 3 \pmod{4}$ we have  

%\[ m_\pi = 
%\begin{cases}
%	0 \mod{4}  &\text{ if } n \text{ is odd }\\
%	1 \mod{4}  &\text{ if } q \equiv 3 \mod 8 \text{ and } n \equiv 2 \mod 4 \\
%	1 \mod 4 & \text{ if }q \equiv 7 \mod 8 \text{ and } n \equiv 0 \mod 4\\
%	-1 \mod{4} &\text{ if } q \equiv 3 \mod 8 \text{ and for } n \equiv 0 \mod 4\\
%	-1 \mod{4} &\text{ if } q \equiv 7 \mod 8 \text{ and for } n \equiv 2 \mod 4
%\end{cases}
%\]

\[ w_2(\pi)=
\begin{cases}
	\sum\limits_{i=1}^n v_i^2,  &\text{ if } q \equiv 3 \mod 8 \text{ and for } n \equiv 0 \pmod 4,\\
	\sum\limits_{i=1}^n v_i^2,  &\text{ if } q \equiv 7 \mod 8 \text{ and for } n \equiv 2 \pmod 4,\\
	0, &\text{  otherwise }.
\end{cases}
\]

\subsubsection{\bf{Computation of $w_4(\pi)$}}

We summerize the computations in the following table.

\begin{center}
	\begin{table}[ht]
		\caption{ Computation of $w_4(\pi)$ where $\pi$ is the Steinberg representation of $\GL_n(\mb F_q)$}
		\centering	
		
		\tabcolsep=.6cm
		\renewcommand{\arraystretch}{2}
		\begin{tabular}{|p{1.5cm}|p{1.5cm}|p{1.5cm}|p{1.5cm}|}
			\hline  
			$q \mod 8$ & $n \mod 4$ & $\binom{m_\pi/2}{2} \mod 2$&$n_\pi/4 \mod 2$\\
			\hline 
			$1$& $0,1,2,3$ & 0 &0\\
			\hline  
			$5$ &
			$0$ & 1 &0\\  
			\hline
			$ 5$ & $ 1$ & 0 &1\\
			\hline
			$ 5$& $ 2$ & 0 &0\\
			\hline
			$ 5$& $ 3$ &1&1\\
			\hline 
			$ 9$ & $1,3$ & 0&0\\
			\hline $ 9$ & $0,2$ & 1 &0\\
			\hline 
			$ 13$ & $ 0$ &0&0\\
			\hline
			$ 13$ & $ 1$ &0&1\\
			\hline
			$ 13$ & $ 2$ &1&0\\
			\hline
			$13$ & $3$ &1&1\\
			\hline 
		\end{tabular}
		\label{Stein.table2}
	\end{table}
\end{center}
\newpage

\section{Data Availability Statement}

Our manuscript has no associated data.

\section{Conflict of Interest Statement}

All authors declare that they have no conflicts of interest.

\bibliographystyle{alpha}
\bibliography{mybib}

\newcommand{\etalchar}[1]{$^{#1}$}
\begin{thebibliography}{GKT89}

\bibitem[AM13]{adem}
Alejandro Adem and R~James Milgram.
\newblock {\em Cohomology of finite groups}, volume 309.
\newblock Springer Science \& Business Media, 2013.

\bibitem[Ben91]{benson}
David~J Benson.
\newblock {\em Representations and cohomology}, volume~2.
\newblock Cambridge university press, 1991.

\bibitem[Car85]{carter}
Roger~William Carter.
\newblock Finite groups of lie type: Conjugacy classes and complex characters.
\newblock {\em Pure Appl. Math.}, 44, 1985.

\bibitem[DI08]{super}
Persi Diaconis and I~Isaacs.
\newblock Supercharacters and superclasses for algebra groups.
\newblock {\em Transactions of the American Mathematical Society},
  360(5):2359--2392, 2008.

\bibitem[FP06]{priddy}
Zbigniew Fiedorowicz and Stewart Priddy.
\newblock {\em Homology of classical groups over finite fields and their
  associated infinite loop spaces}, volume 674.
\newblock Springer, 2006.

\bibitem[GJ21]{ganjo1}
Jyotirmoy Ganguly and Rohit Joshi.
\newblock Spinorial representations of orthogonal groups.
\newblock {\em J. Lie Theory}, 31(1):265--286, 2021.

\bibitem[GJ22]{ganjo}
Jyotirmoy Ganguly and Rohit Joshi.
\newblock Stiefel whitney classes for real representations of gl2(fq).
\newblock {\em International Journal of Mathematics}, page 2250010, 2022.

\bibitem[GKT89]{guna}
J~Gunarwardena, B~Kahn, and C~Thomas.
\newblock Stiefel-whitney classes of real representations of finite groups.
\newblock {\em Journal of Algebra}, 126(2):327--347, 1989.

\bibitem[GS20]{jyoti}
Jyotirmoy Ganguly and Steven Spallone.
\newblock Spinorial representations of symmetric groups.
\newblock {\em Journal of Algebra}, 544:29--46, 2020.

\bibitem[JS21]{spjoshi}
Rohit Joshi and Steven Spallone.
\newblock Spinoriality of orthogonal representations of gl(n,q).
\newblock {\em Pacific Journal of Mathematics}, 311(2):369--383, 2021.

\bibitem[Lan02]{lang}
Serge Lang.
\newblock {\em Graduate Texts in Mathematics: Algebra}.
\newblock Springer, 2002.

\bibitem[Mit01]{mitchell}
Stephen~A Mitchell.
\newblock Notes on principal bundles and classifying spaces.
\newblock {\em Lecture Notes. University of Washington}, 2001.

\bibitem[Nak60]{naka}
Minoru Nakaoka.
\newblock Decomposition theorem for homology groups of symmetric groups.
\newblock {\em Annals of Mathematics}, pages 16--42, 1960.

\bibitem[T{\etalchar{+}}87]{toda}
Hiroshi Toda et~al.
\newblock Cohomology of classifying spaces.
\newblock In {\em Homotopy theory and related topics}, pages 75--108.
  Mathematical Society of Japan, 1987.

\end{thebibliography}

\end{document}